\documentclass[12pt]{article}
\usepackage{amsfonts}
\usepackage[T1]{fontenc}
\usepackage[utf8]{inputenc}
\usepackage[letterpaper]{geometry}
\usepackage{amsmath, mathrsfs}
\usepackage{graphicx}
\usepackage[active]{srcltx}
\numberwithin{equation}{section}

\usepackage{amsmath}
\usepackage{graphicx}
\usepackage[active]{srcltx}
\usepackage{amssymb}
\usepackage{color}
\usepackage{amsmath}
\usepackage{graphicx}
\usepackage[active]{srcltx}
\include{srctex}
\newtheorem{theorem}{Theorem}[section]

\newtheorem{corollary}[theorem]{Corollary}

\newtheorem{definition}[theorem]{Definition}

\newtheorem{lemma}[theorem]{Lemma}

\newtheorem{proposition}[theorem]{Proposition}
\newtheorem{remark}[theorem]{Remark}

\newenvironment{proof}[1][Proof]{\textbf{#1.} }{\ \rule{0.5em}{0.5em}}

\def\E{\mathbb{E}}
\def\L{\mathscr{L}}
\def\R{\mathbb{R}}

\def\ch{\mathcal{H}}\def\bh{\mathbb{H}}
\def\e{\varepsilon}
\def\d{\delta}

\def\Var{\text{Var}}
\begin{document}

\title{Large deviations for functionals of some self-similar Gaussian processes}
\date{}
\author{Xiaoming Song}
\maketitle
\begin{abstract}
We prove  large deviation principles for $\int_0^t \gamma(X_s)ds$, where $X$ is a $d$-dimensional self-similar Gaussian process and $\gamma(x)$ takes the form of the Dirac delta function $\delta(x)$, $|x|^{-\beta}$ with $\beta\in (0,d)$, or $\prod_{i=1}^d |x_i|^{-\beta_i}$ with $\beta_i\in(0,1)$. In particular, large deviations are obtained for the functionals of $d$-dimensional fractional Brownian motion, sub-fractional Brownian motion and bi-fractional Brownian motion. As an application, the critical exponential integrability of the functionals is discussed. 


\end{abstract}

\section{Introduction}  

The large deviation principles for functionals of symmetric L\'evy stable  processes such as the (intersection) local time and  Riesz potentials of additive processes were studied in  \cite{chen07, bcr}, where the properties of  symmetric L\'evy stable  process (including the standard Brownian motion) such as self-similarity and independent increment property play a crucial role in the analysis. Later, exact forms of large deviations for (intersection) local times of fractional Brownian motion (fBm for short) and the Riemann-Liouville process were obtained in \cite{clrs}. The results in \cite{clrs} are surprising and are not a ``natural'' extension of  \cite{chen07, bcr}, in the sense that fBm and the Riemann-Liouville process are not Markovian and the techniques for L\'evy processes do not apply. 

Let $\gamma(x)$ be one of the following functions: the Dirac delta function $\delta(x)$, $|x|^{-\beta}$ with $\beta\in (0,d)$,  and $\prod_{i=1}^d |x_i|^{-\beta_i}$ with $\beta_i\in(0,1)$. Throughout the article, we use the convention that  $\beta=d$ if $\gamma(x)=\delta(x)$ and that $\beta=\sum_{i=1}^d \beta_i$ if $\gamma(x)=\prod_{i=1}^d |x_i|^{-\beta_i}$. Under this convention, the functional $\gamma$ has homogeneity, i.e., $\gamma(ax)=a^{-\beta}\gamma(x)$ for $a>0$ and $x\in \mathbb{R}^d$.

 This article concerns the large deviations  for $\int_0^t \gamma(X_s)ds$, where $X$ is a $d$-dimensional  self-similar Gaussian process satisfying some conditions.  In particular, the large deviations for the functionals of fractional Brownian motion  $B^H$ (fBm for short) with $H\in(0,1)$  and $H\beta<1$, sub-fractional Brownian motion $S^H$ (sub-fBm for short) with $H<\frac12$ and $H\beta<1$, and  bi-fractional Brownian motion $Z^{H, K}$ (bi-fBm for short) with $H\in(0,1), K\in (0,1]$ and  $HK\beta<1$. 
 
 Instead of carrying out  a direct analysis for the functionals of fBm, sub-fBm and bi-fBm, we first obtain the large deviations for the Riemann-Liouville process $\int_0^t (t-s)^{H-\frac12}dW_s$, where $W$ is $d$-dimensional standard Brownian motion. In light of Lemma \ref{lemma10},  the large deviation principle for the functional of the Riemann-Liouville process is reduced to proving the existence of the limit for the log moments of the functional, for which it suffices to show the sub-additivity (see Propositions \ref{prop-sub-additive} and \ref{prop-3.2}).  After we obtain the results for the functionals of the Riemann-Liouville process, we study the large deviations for the functionals of fBm, sub-fBm and bi-fBm  
  by comparing them with the functionals of the Riemann-Liouville process. The comparison strategy was initially developed in \cite{clrs}, and we briefly interpret the two key ingredients of the idea below.

Firstly, we observe that, for general Gaussian processes $X$ and $Y$ which both possess certain self-similarity, if $Y=X+\eta$ such that $X$ and $\eta$ are independent, and $\eta$ belongs to the Cameron-Martin space of $X$ almost surely, then the large deviations of $X$ and $Y$, if one of them exists, both exist and coincide with each other (see Proposition \ref{prop-ld} for details). The crucial condition here is that $\eta$ belongs to the Cameron-Martin space of $X$, which yields that, conditioned on $\eta$, the distributions of $X$ and $Y$ are equivalent. An heuristic explanation for the coincidence of the large deviations for the functionals of $X$ and $Y$ is that,  $\eta$ is  ``regular'' enough in comparison with $X$, and thus the perturbation of $\eta$ is just negligible. 
 
The second key ingredient in the comparison strategy is to show that the decompositions for fBm, sub-fBm (\cite{ct}) and bi-fBm (\cite{ln})   satisfy the conditions in Proposition \ref{prop-ld}, for which one needs to characterize the Cameron-Martin spaces for fBm, sub-fBm and bi-fBm (see Section \ref{section4.1}).

This article is organized as follows. In Section 2,  the comparison principle for the large deviations of functionals of  self-similar Gaussian processes is developed in a general context. 
Large deviations for the functionals of the Riemann-Liouville process are obtained in section 3.   Finally, section 4 is devoted to the study of large deviations for functionals of fBm, sub-fBm and bi-fBm.

\section{Large deviations by comparison}\label{section2}
 Suppose that $X$ and $Y$ are two $d$-dimensional self-similar Gaussian processes such that $\{X_{at}, t\ge 0\}\overset{d}{=} \{a^\alpha X_t, t\ge 0\}$ and  $\{Y_{at}, t\ge 0\}\overset{d}{=} \{a^\alpha Y_t, t\ge 0\}$ with $\alpha>0$ and $a>0$, and that $Y\overset d= X+\eta$ where  $\eta$ is a Gaussian process which is independent of $X$ and belongs to the reproducing kernel Hilbert space of $X$ almost surely. Let $\gamma(x)$ take the forms of the Dirac delta function $\delta(x)$ with $\alpha d<1$, $\prod_{i=1}^d|x_i|^{-\beta_i}$ with $\beta_i\in(0,1)$ and $\alpha\sum_{i=1}^d \beta_i<1$, or $|x|^{-\beta}$ with $\beta\in(0,d)$ and $\alpha \beta<1$.  
 
 The major goal of  this section is to prove the following equality under some conditions,  $$\lim_{t\to\infty}\frac1t\log \E\exp\left(\int_0^t \gamma(X_s)ds\right)=\lim_{t\to\infty}\frac1t\log \E\exp\left(\int_0^t \gamma(Y_s)ds\right).$$ 
 This result is useful, for instance, to derive the large deviations for fBm, sub-fBm and  bi-fBm (see Section \ref{section4}). 
\subsection{Reproducing kernel Hilbert spaces}
In this subsection, we summarize some preliminaries on reproducing kernel Hilbert spaces associated with Gaussian processes. We refer readers to \cite{Janson} for more details. 

In a probability space $(\Omega, \mathcal F, P)$, consider a one-dimensional centered Gaussian process $X=\{X_t, 0\le t\le T\}$ with covariance function 
$$R(s,t)=\E[X_sX_t],\ 0\leq s\leq t\leq T. $$

The {\it reproducing kernel Hilbert space (RKHS)} associated with the Gaussian process $X$,  denoted by $\bh(X)$, is the completion of the linear span of the functions $\sum_{i=1}^n a_i R(s_i, \cdot)$ with $n\in \mathbb N, a_i\in\R, s_i\in [0,T]$, $i=1,\dots, n$, under the norm induced by the inner product
$$\left\langle R(s, \cdot),~ R(t, \cdot)\right\rangle_{\bh(X)}=  R(s, t). $$
Note that the RKHS is also referred to as the {\it Cameron-Martin space} (\cite[Theorem 8.15]{Janson}), and in this article we do not distinguish these two terminologies. 

As a comparison, we also recall the space of integrands $\dot f$ for Wiener integrals with respect to $X$, denoted  by $\ch(X)$, which is defined as the completion of the linear span of the simple functions $\sum_{i=1}^n a_i \mathbf{1}_{(s_i, t_i]}$ under the norm induced by the inner product
$$\left\langle \mathbf{1}_{(0,s]}, \mathbf{1}_{(0,t]}\right\rangle_{\ch(X)}= R(s,t).$$
 Denote by $X(\dot f)$ the Wiener integral for $\dot f\in \ch(X)$. The collection of these Wiener integrals is the {\it first Wiener chaos} $\mathbf H_1$ of $X$ (see, e.g., \cite{nualart06}). Then $\E[X(\dot f)X(\dot g)]=\langle \dot f, \dot g\rangle_{\ch(X)}$, for $\dot f,\dot g\in \ch(X)$. Furthermore, setting $f(t)=\E[X(\dot f) X_t]$, we have $f\in \bh(X)$, and $\langle f, g\rangle_{\bh(X)}=\langle \dot f, \dot g\rangle_{\ch(X)}.$  Therefore,  the RKHS $\bh(X)$, the space   $\ch(X)$ of  integrands of Wiener integrals,  and the first Wiener chaos $\mathbf H_1$ of $X$ are isometric to each other.     For example,  when $X=W$ is a Brownian motion, $\ch(W)=L^2[0,T]$, $\bh(W)=\{\int_0^\cdot \dot f(s) ds, \dot f(s)\in L^2[0,T]\}$  and $\mathbf H_1=\{W(\dot f)=\int_0^T \dot f(s) dW_s, \dot f\in L^2[0,T]\}$.
 
 One  important feature of the RKHS  is the following. For a function $h:[0,T]\to \R$, the laws of $X+h$ and  $X$  are mutually absolutely continuous (resp. mutually singular) if $h\in \bh(X)$ (resp. if $h\notin \bh(X)$), see, e.g., \cite[Theorem 14.17]{Janson}. Moreover, for $h\in \bh(X)$, by the Cameron-Martin theorem,  the measure $\tilde P$ defined by
 $$ \frac{d\tilde P}{dP}=\exp\left(-X(\dot h)-\frac12\|h\|_{\bh(X)}^2\right)$$
is a probability measure, under which $X+h$ has the same distribution as $X$ under $P$.

\subsection{Preliminaries on large deviation principles}
In this subsection, let $L=\{L_t, t\geq 0\}$ be a stochastic process with non-negative values. We will recall some results on the large deviation principle for the process $L$.
\begin{definition}\label{def-ldp}
A function $I:\,\mathbb{R^+}\to [0,\infty]$ is called a rate function on $\mathbb{R^+}$, if for each $M<\infty$ the level set $\{x\in\mathbb{R}^+:\,I(x)\leq M\}$ is a closed subset of $\mathbb{R}^+$. If the level set $\{x\in\mathbb{R}^+:\,I(x)\leq M\}$ is compact for any $M<\infty$, then $I(\cdot)$ is said to be a good rate function. For any $A\in\mathcal{B}(\mathbb{R}^+)$, we define $I(A)=\inf_{x\in A}I(x)$.
\end{definition}
\begin{definition}\label{ldp}
Let $I(\cdot)$ be a rate function on $\mathbb{R}^+$, and let $\{b(t),t\geq 0\}$ be  a sequence of positive real numbers such that $\lim\limits_{t\to\infty}b(t)=\infty$. The stochastic process $L$ is said to satisfy the large deviation principle with speed $\{b(t)\}$ and rate function $I(\cdot)$ if the following two conditions hold:
\[
\limsup\limits_{t\to \infty}\frac1{b(t)}\log\mathbb{P}(L_t\in F)\leq -\inf\limits_{\lambda\in F}I(\lambda), \  \mbox{ for any closed set } F\subseteq \mathbb{R}^+,
\]
and
\[
\liminf\limits_{t\to \infty}\frac1{b(t)}\log\mathbb{P}(L_t\in G)\geq -\inf\limits_{\lambda\in G}I(\lambda), \  \mbox{ for any open set } G \subseteq \mathbb{R}^+.
\]
\end{definition}
The following result shows that under some mild conditions on the rate function $I(\cdot)$, the large deviation principle defined above is equivalent to the asymptotic behavior of tail properties (see \cite[Theorem 1.2.1]{chen}).
\begin{theorem}\label{equiv}
Suppose that the rate function $I(\cdot)$ is strictly increasing and continuous on $\mathbb{R}^+$. The following two statements are equivalent:
\begin{itemize}
\item[(a)] The large deviation principle given in Definition \ref{ldp} holds.
\item[(b)] For any $\lambda>0$, 
\begin{equation}\label{log-prob}
\lim\limits_{t\to \infty}\frac1{b(t)}\log\mathbb{P}(L_t\geq \lambda)=-I(\lambda).
\end{equation}
\end{itemize}
\end{theorem}
\begin{definition}
A convex function $\Lambda(\theta): \mathbb{R}^+\to [0,\infty]$ is said to be essentially smooth on $\mathbb{R}^+$, if
\begin{itemize}
\item[(1)] there is a $\theta_0>0$ such that $\Lambda(\theta)<\infty$ for every $\theta\in [0,\theta_0]$.
\item[(2)] the function $\Lambda(\cdot)$ is differentiable in the interior $D_\Lambda^o=(0,a)$ ($0<a\leq \infty$) of the domain $D_\Lambda=\{\theta\in\mathbb{R}^+: \Lambda(\theta)<\infty\}$.
\item[(3)] the function $\Lambda(\cdot)$ is steep at the right end of the domain and is flat at the left end of the domain, i.e.,
\[
\lim\limits_{\theta\to a^-}\Lambda'(\theta)=\infty \ \mbox{ and }\ \Lambda'(0^+)=\lim\limits_{\theta\to 0^+}\frac{\Lambda(\theta)-\Lambda(0)}{\theta}=0.
\]
\end{itemize}
\end{definition}

The following result appeared in \cite[Theorem 1.2.4]{chen} is a version of the G\"{a}rtner-Ellis large deviation.
\begin{theorem}\label{thm-GE}
Assume that for all $\theta\geq 0$, the limit
\[
\Lambda(\theta)=\lim\limits \frac{1}{b(t)}\log\E\exp\{\theta b(t)L_t\}
\]
exists as an extended real number, and that the function $\Lambda(\cdot)$ is essentially smooth on $\mathbb{R}^+$. Then, the function 
\[
I(\lambda)=\sup\limits_{\theta>0}\left\{\theta\lambda-\Lambda(\theta)\right\}, \ \lambda\geq 0
\]
is strictly increasing and continuous on $\mathbb{R}^+$. Moreover, the large deviation principle in Definition \ref{ldp} and equation \eqref{log-prob} hold and they are equivalent.
\end{theorem}

As the converse of the G\"{a}rtner-Ellis theorem, we have the following Varadhan's integral lemma (see \cite[Theorem 1.1.6]{chen}).
\begin{lemma}[Varadhan's integral lemma]\label{lem-vil}
Assume that the stochastic process $L$ satisfy the large deviation principle with speed $\{b(t), t\geq 0\}$ and a good rate function $I(\cdot)$. Let $\phi:\mathbb{R}^+\to \mathbb{R}$ be any continuous function. Suppose that, for some $\rho>1$, the following condition holds
\[
\limsup\limits_{t\to\infty}\frac{1}{b(t)}\log\mathbb{E}\exp\left\{\rho b(t)\phi(L_t)\right\}<\infty,
\]
then we have
\[
\lim\limits_{t\to\infty}\frac{1}{b(t)}\log\mathbb{E}\exp\left\{ b(t)\phi(L_t)\right\}=\sup\limits_{\lambda\in\mathbb{R}^+}\left\{\phi(\lambda)-I(\lambda)\right\}.
\]
\end{lemma}
\subsection{Comparison strategy}
We first validate the definition of $\int_0^t \delta(X_s)ds$ for a class of Gaussian process. Denote the heat kernel on  $\mathbb{R}^d$ by $\displaystyle p_\varepsilon(x)=(2\pi \varepsilon)^{-\frac{d}2}e^{-\frac{|x|^2}{2\varepsilon}}$. 

\begin{proposition}\label{lemma-lt}
Let $\{X_t=(X_t^1, \dots, X_t^d), 0\le t\le T\}$ be a centered Gaussian process, the components of which are independent and have the same distribution. If there exist constants $C_T>0$ and  $0<\alpha<1/d$, such that, for $0\le r\le s\le T$,
\begin{equation}\label{con1}
\Var(X_r^1)\ge C_T r^{2\alpha} \text{ and } \Var(X_s^1|X_r^1)=\mathbb{E}\left(\left[X_s^1-\mathbb{E}\left(X_s^1\big|X_r^1\right)\right]^2\big|X_r^1\right)\ge C_T (s-r)^{2\alpha},
\end{equation}
then $\int_0^T p_\varepsilon(X_t)dt$ converges in $L^2$ as $\varepsilon$ goes to $0$.  The limit is denoted by $L_T(X):=\int_0^T\delta(X_s)ds$ and called the local time of $X$.
\end{proposition}
\begin{proof}
It suffices to show that the sequence $\E\left[\int_0^T p_\e(X_r)dr\int_0^T p_\d(X_s) ds\right]$ converges to the same limit as $\e$ and $\d$ go to zero. Note that the Fourier transform of $p_\e(x)$ is $$\hat p(\xi)=\int_{\mathbb{R}^d}e^{-ix\cdot \xi}p_\varepsilon(x)dx=e^{-\e|\xi|^2/2},$$ 
and also note that the inverse Fourier transform implies $$p_\e(x)=(2\pi)^{-d/2}\int_{\R^d} e^{ix\cdot \xi}e^{-\e|\xi|^2/2}d\xi, $$
where $dx=dx_1\dots dx_d$ and $d\xi=d\xi_1\dots d\xi_d.$

Then for fixed $ r, s\in [0,T]$, we have
\begin{align*}
&\E[p_\e(X_r)p_\d(X_s)]\\
=&(2\pi)^{-d}\int_{\R^{2d}} \exp\left(-\frac12(\e|\xi|^2+\d|\eta|^2)\right) \E \exp\left(i(X_r\cdot \xi+X_s\cdot \eta)\right) d\xi d\eta\\
=& \left((2\pi)^{-1} \int_{\R^2}  \exp\left(-\frac12(\e \xi^2+\d\eta^2)\right) \E \exp\left(i(X_r^1 \xi+X_s^1 \eta)\right)  d\xi d\eta\right)^d\\
=& \left((2\pi)^{-1} \int_{\R^2}  \exp\left(-\frac12(\e \xi^2+\d\eta^2)\right)  \exp\left(-\frac12\Var(X_r^1 \xi+X_s^1 \eta)\right)  d\xi d\eta\right)^d\\
=& \left((2\pi)^{-1} \int_{\R^2}  \exp\left(-\frac12(\e \xi^2+\d\eta^2)\right)  \exp\left(-\frac12 (\xi,\eta) Q(r,s) (\xi, \eta)^T\right) d\xi d\eta\right)^d,
\end{align*}
where $Q(r,s)$ is the covariance matrix of $(X_r^1, X_s^1).$ It is well known (see, e.g., \cite{Berman} or \cite[Lemma 3.8]{clrs}) that $\det Q(r,s)=\Var(X_r^1)\Var(X_s^1|X_r^1)=\Var(X_s^1)\Var(X_r^1|X_s^1)$, and hence,  by \eqref{con1}, $$\det Q(r,s)\ge C_T^2 (r\wedge s)^{2\alpha} |r-s|^{2\alpha}.$$ 

Therefore, by the dominated convergence theorem, we can get 
$$\lim_{\e,\d\to0}\E[p_\e(X_r)p_\d(X_s)]=\left(\det Q(r,s)\right)^{-d/2}\le C_T^{-d} (r\wedge s)^{-\alpha d} |r-s|^{-\alpha d}. $$

Since $\alpha d<1$, we obtain $\int_0^T\int_0^T(r\wedge s)^{-\alpha d} |r-s|^{-\alpha d}drds<\infty$. Then one can apply  the dominated convergence theorem to deduce 
\begin{align*}
\lim_{\e,\d\to0}\E\left[\int_0^T p_\e(X_r)dr\int_0^T p_\d(X_s) ds\right]&=\int_0^T\int_0^T \left(\det Q(r,s)\right)^{-d/2} drds. 
\end{align*}
The proof is concluded.\hfill 
\end{proof}
\begin{remark}\label{rmk-fbm-sub-bi}
If the conditions in \eqref{con1} are satisfied, we say that the Gaussian process $X$ has local nondeterminism. In particular,   when $Hd<1$, $HKd<1$, by Proposition \ref{lemma-lt} and the local nondeterminism of fBm $B^H$, sub-fBm $S^H$ and bi-fBm $Z^{H,K}$ (see \cite{berman, luan, tx}), the local times $L_t(B^H)$,  $L_t(S^H)$ and  $L_t(Z^{H,K})$ exist. 
\end{remark}
\begin{remark}\label{rmk-RL}
Let $X^\alpha_t=\int_0^t (t-s)^{\alpha-\frac12}dW_s$ be the $1$-dimensional Riemann-Liouville process, where $W$ is a standard Brownian motion. Then, we can show that
\begin{itemize}
\item[(i)]
$
Var(X_t^\alpha)=\int_0^t(t-s)^{2\alpha-1}ds=\frac{t^{2\alpha}}{2\alpha}\ ;
$
\item[(ii)] for any $0\leq r<t<\infty$, 
\[
\mathbb{E}\left(X_t^\alpha\big|\mathcal{F}_r^W\right)=\mathbb{E}\left(\int_0^t(t-s)^{\alpha-\frac12}dW_s\big|\mathcal{F}_r^W\right)=\int_0^r(t-s)^{\alpha-\frac12}dW_s,
\]
and
\begin{align*}
Var(X_t^\alpha\big|X_r^\alpha) \geq&\ \E\left(Var(X_t^\alpha\big|\mathcal{F}_r^W)\big|X_r^\alpha\right)\\
=&\ \E\left(\left.\E\left(\left.\left[\int_r^t(t-s)^{\alpha-\frac12}dW_s\right]^2\right|\mathcal{F}_r^W\right)\right|X_r^\alpha\right)\\
=&\ \frac{1}{2\alpha}(t-r)^{2\alpha};
\end{align*}
\item[(iii)] for any $0\leq r<t<\infty$, by some changes of variables 
\begin{align*}
&Var(X_t^\alpha-X_r^\alpha)\\
=&\ \E\left(\int_0^r\left[(t-s)^{\alpha-\frac12}-(r-s)^{\alpha-\frac12}\right]dW_s\right)^2+\E\left(\int_r^t(t-s)^{\alpha-\frac12}dW_s\right)^2\\
=&\ \int_0^r\left[(t-s)^{\alpha-\frac12}-(r-s)^{\alpha-\frac12}\right]^2ds+\int_r^t(t-s)^{2\alpha-1}ds\\
=&\ (t-r)^{2\alpha}\int_0^{r/(t-r)}\left[(1+u)^{\alpha-\frac12}-u^{\alpha-\frac12}\right]^2du+\frac{1}{2\alpha}(t-r)^{2\alpha}\\
\leq&\ (t-r)^{2\alpha}\int_0^{\infty}\left[(1+u)^{\alpha-\frac12}-u^{\alpha-\frac12}\right]^2du+\frac{1}{2\alpha}(t-r)^{2\alpha}\\
=&\ C_\alpha(t-r)^{2\alpha},
\end{align*}
where $C_\alpha=\int_0^{\infty}\left[(1+u)^{\alpha-\frac12}-u^{\alpha-\frac12}\right]^2du+\frac{1}{2\alpha}$.
\end{itemize}
From $(i)$ and $(ii)$, we see that the Riemann-Liouville process has local nondeterminism.  An estimate for the variance of the increment of this process is given in $(iii)$.
\end{remark}

In the following sections, we always assume that the process $\{X_t, 0\leq t\leq 1\}$ can be viewed as a Gaussian random vector in a separable Banach space. The result below is an important property of Gaussian measure  (see e.g. \cite[Lemma 3.7]{clrs}). 
\begin{lemma}\label{max-principle}
Suppose $\mu$ is a centered Gaussian measure on a separable Banach space $B$. Let $\mathbb H_\mu$ denote the RKHS of $\mu$, and let $h:B\mapsto \R^+$ be a symmetric measurable function ($h(-x)=h(x)$ for any $x\in B$). Then, for every $y$ in $\mathbb H_\mu$,  we have
$$\int_B h(x+y) \mu(dx) \ge \exp\left(-\frac12 \|y\|_{\mathbb H_\mu}^2\right)\int_B h(x) \mu(dx),$$
where $\|y\|_{\mathbb H_\mu}$ is the norm of $y$ in $\mathbb H_\mu.$
\end{lemma}

The inequalities stated in the following will be used in the proof of Proposition \ref{prop-ld} and the sub-additive property for the Riemann-Liouville process.
\begin{lemma}\label{lem-max-principle}
Let $\gamma$ be a tempered distribution on $\R^d$ with its Fourier transform $\nu(dx)$ being a non-negative measure on $\R^d$, i.e., $\gamma$ is a non-negative definite distribution. Then for any centered Gaussian random vector $X\sim N(0, \Sigma)$, where $\Sigma$ is a positive definite matrix in $\mathbb{R}^{d\times d}$,  we have 
$$\E[\gamma(X+a)]\le \E[\gamma(X)], \text{ for all } a\in \R^d.$$
\end{lemma}
\begin{proof}
Denote  the probability density function of the Gaussian random vector $X$ by $p_{\Sigma}(x)$.  Note that $p_{\Sigma}(x)$ belongs to the Schwartz space $\mathcal S(\R^d)$ and its Fourier transform is $\hat p_{\Sigma}(\xi)=\exp\{-\frac{\xi^T\Sigma \xi}{2}\}$. Then
\begin{align*}
&\E[\gamma(X+a)]=\int_{\R^d} \gamma(x+a) p_{\Sigma}(x) dx \\
&=\int_{\R^d} \hat p_{\Sigma}(\xi) e^{-ia\cdot \xi}\nu(d\xi)\le \int_{\R^d} \hat p_{\Sigma}(\xi) \nu(d\xi)=\E[\gamma(X)]. 
\end{align*}
We complete the proof.\hfill
\end{proof}
\begin{remark}\label{rem-max-principle}
Similarly, one can show that for a centered Gaussian vector $(X_1,\dots, X_n)$ with $X_i$ being a $d$-dimensional Gaussian vector, we have for any $a=(a_1,\dots, a_n)\in \R^{d\times n}$
$$\E\left[\prod_{i=1}^n \gamma(X_i+a_i)\right]\le \E\left[\prod_{i=1}^n\gamma(X_i)\right].$$ 
When $\gamma$ is a measurable function which is also symmetric ($\gamma(-x)=\gamma(x)$), and the result can also be obtained by \cite[Lemma 3.7 (i)]{clrs}.
\end{remark}

The following proposition is the main result in this subsection. 
\begin{proposition}\label{prop-ld}
Let $\{X_t, t\ge 0\}$, $\{Y_t, t\ge 0\}$ and $\{\eta_t, t\ge 0\}$ be $d$-dimensional centered Gaussian processes satisfying the following conditions:
\begin{itemize}
\item [(i)] there exists $\alpha>0$ such that $\{X_{at},t\ge0\}\overset{d}{=}a^{\alpha} \{X_t, t\ge 0\}$ and $\{Y_{at}, t\ge 0\}\overset{d}{=}a^{\alpha} \{Y_t, t\ge 0\}$ for any $a>0$;
\item[(ii)] $Y\overset{d}{=}X+\eta$;
\item[(iii)] $X$ and $\eta$ are independent;
\item[(iv)] for any $\e \in(0,1) $, there exists a process  $\eta^\e$ such that $\eta_t^\e=\eta_t$ for $t\ge \e$, and  $\{\eta_t^\e, t\in[0,1]\}$ belongs to the RKHS of $\{X_t, t\in[0,1]\}$ almost surely.
\end{itemize}
 If either $\lim_{t\to \infty}\frac1t \log \E\exp\left(\int_0^t \gamma(X_s)ds\right)  \text{ or }  \lim_{t\to \infty}\frac1t \log \E\exp\left(\int_0^t \gamma(Y_s)ds\right)$ exits as a finite number, 
then both limits exist and are equal to each other.
\end{proposition}
\begin{proof}
Without loss of generality, we assume $\lim\limits_{t\to \infty}\frac1t \log \E\exp\left(\int_0^t \gamma(X_s)ds\right)$ exists.

 It follows from  Remark \ref{rem-max-principle} that $$\E\left(\int_0^t \gamma(Y_s)ds\right)^n\le \E\left(\int_0^t \gamma(X_s)ds\right)^n, $$
and hence
$$\limsup_{t\to \infty}\frac1t \log \E\exp\left(\int_0^t \gamma(Y_s)ds\right)\le \lim_{t\to \infty}\frac1t \log \E\exp\left(\int_0^t \gamma(X_s)ds\right). $$

To get the desired result, we shall prove the opposite direction of the above inequality with $\limsup$ replaced by $\liminf$. Fixing an arbitrary $\e\in(0,1),$ and denoting $Y^\e=X+\eta^\e$, by Lemma \ref{max-principle}, Minkowski's inequality and the scaling property $\int_0^a \gamma (X_s) ds\overset{d}= a^{1-\alpha\beta} \int_0^1 \gamma (X_s)ds$,  we have
\begin{align}\label{eq3.2}
\E\left(\int_0^1 \gamma(Y_s)ds\right)^n& \ge  \E\left(\int_\e^1 \gamma(Y_s^\e)ds\right)^n \notag\\
&\ge \E\exp\left(-\frac12 \|\eta^\e\|^2_{\bh(X) } \right) \E\left(\int_\e^1 \gamma(X_s)ds\right)^n \notag
\\
&= A_\e~ \E\left(\int_0^1 \gamma(X_s)ds-\int_0^\e \gamma(X_s)ds\right)^n\notag\\
&\ge A_\e~\left(\left(\E\left(\int_0^1 \gamma(X_s)ds\right)^n\right)^{1/n}-\left(\E\left(\int_0^\e \gamma(X_s)ds\right)^n\right)^{1/n}\right)^n\notag\\
&=  A_\e~\left( 1-\e^{1-\alpha \beta}\right)^n \E\left(\int_0^1 \gamma(X_s)ds\right)^n,
\end{align}
where $ \|\eta^\e \|_{\bh(X)}<\infty$ a.s. is the norm endowed in the RKHS of $\{ X_s, s\in[0,1] \}$ and $A_\e=\E\exp\left(-\frac12 \|\eta^\e\|^2_{\bh(X)} \right)\in(0,1]$ is independent of $n$.

 Thus, by the scaling property for the functional of $Y$ and  \eqref{eq3.2}, we obtain
\begin{align}\label{eq4.3}
\E\exp\left(\int_0^{t}\gamma(Y_s)ds\right) &=\sum_{n\ge0} \frac1{n!}t^{n(1-\alpha \beta)} \E\left(\int_0^{1}\gamma(Y_s)ds\right)^n\notag\\
&\ge A_\e \sum_{n\ge0}\frac1{n!}  \left(t^{1-\alpha \beta}(1-\e^{1-\alpha \beta})\right)^n \E\left(\int_0^{1}\gamma(X_s)ds\right)^n\notag\\
&=A_\e\,\E\exp\left(t^{1-\alpha \beta}(1-\e^{1-\alpha \beta})\int_0^{1}\gamma(X_s)ds\right).
\end{align}

For random variables $F$, $G$ with $\E e^{\theta F}<\infty$ and $\E e^{\theta G}<\infty$ for all $\theta>0,$  H\"older's inequality yields 
$$ \log\E e^{F-G}\ge p\log \E e^{\frac Fp} -\frac pq \log \E e^{\frac qp G}, $$
where $p,q >1$ and $1/p+1/q=1.$ Applying this inequality to the right-hand side of \eqref{eq4.3} and using the scaling property $\int_0^a \gamma (X_s) ds\overset{d}= a^{1-\alpha\beta} \int_0^1 \gamma (X_s)ds$, we get 
\begin{align*}
&\log\E\exp\left(\int_0^{t}\gamma(Y_s)ds\right)\\
\ge& \log A_\e + p\log \E\exp\left(p^{-1} t^{1-\alpha \beta}\int_0^1 \gamma(X_s)ds\right)-\frac pq\log \E\exp\left(p^{-1}q (\e t)^{1-\alpha \beta}\int_0^1 \gamma(X_s)ds\right)\\
\ge& \log A_\e + p\log \E\exp\left(\int_0^{t\,p^{-(1-\alpha \beta)^{-1}}} \gamma(X_s)ds\right)-\frac pq\log \E\exp\left(\int_0^{\e t\,(p^{-1}q)^{(1-\alpha \beta)^{-1}}} \gamma(X_s)ds\right), 
\end{align*}
and hence 
\begin{align*}
&\liminf_{t\to \infty}\frac1t\log\E\exp\left(\int_0^{t}\gamma(Y_s)ds\right)\\
\ge& \lim_{t\to\infty} \frac pt \log \E\exp\left(\int_0^{t\,p^{-(1-\alpha \beta)^{-1}}} \gamma(X_s)ds\right)-\lim_{t\to\infty}  \frac p{qt}\log \E\exp\left(\int_0^{\e t\,(p^{-1}q)^{(1-\alpha \beta)^{-1}}} \gamma(X_s)ds\right)\\
=&\left( p^{1-(1-\alpha \beta)^{-1}} -\e (pq^{-1})^{1-(1-\alpha \beta)^{-1}}\right)  \lim_{t\to\infty} \frac 1t \log \E\exp\left(\int_0^{t} \gamma(X_s)ds\right).
\end{align*}
Since $\e >0$ can be arbitrarily small and $p$ can be arbitrarily close to $1$,  we obtain
\begin{align*}
\liminf\limits_{t\to \infty}\frac1t\log\E\exp\left(\int_0^{t}\gamma(Y_s)ds\right)
&\ge \lim\limits_{t\to \infty}\frac1t\log\E\exp\left(\int_0^{t}\gamma(X_s)ds\right).
\end{align*}
The proof is completed.
\hfill
\end{proof}
\begin{remark}
It is obvious that  if condition (iv) is replaced by
\begin{itemize}
\item[(iv')] $\{\eta_t,  t\in[0, 1]\}$ belongs to the RKHS of $\{X_t, t\in[0,1]\}$ almost surely, 
\end{itemize}
the result of Proposition \ref{prop-ld} still holds. 
\end{remark}

\section{Large deviations for the functionals of Riemann-Liouville process}

In this section, we let $X^\alpha=\{X^\alpha_t, t\geq 0\}$ be the $d$-dimensional Riemann-Liouville process with parameter $\alpha\in (0,1)$, i.e.,  $X^\alpha_t=\int_0^t (t-s)^{\alpha-\frac12}dW_s$, where $\{W_t,t\geq 0\}$ is a $d$-dimensional Brownian motion. This section is devoted to deriving the large deviations for $\int_0^t \gamma(X_s^\alpha)ds$, where  $\gamma$ is the functional given in Section \ref{section2}.  

\begin{proposition}\label{prop-sub-additive}
 Suppose $X_t=\int_0^t K(t-s) dW_s$, where $K(s):\mathbb{R}^+\to \mathbb{R}^d$ is a measurable function such that $\int_0^T |K(s)|^2ds<\infty$ for all $T>0$. Let $\gamma$ be a tempered distribution on $\R^d$ with its Fourier transform $\nu(dx)$ being a non-negative measure on $\R^d$, i.e., $\gamma$ is a non-negative definite distribution. Then $\log \frac1{m!}\E \left(\int_0^\tau \gamma(X_s)ds \right)^m$ is sub-additive in $m$, where $\tau$ is an exponential time with parameter 1 independent of $X$.
\end{proposition}
\begin{proof}
Denote $[0,t]_{<}^m=[0<s_1<s_2<\dots<s_m <t]$ and $\R_{+,<}^m=[0<s_1<s_2<\dots<s_m<\infty)$. Notice that
\begin{align*}
\frac1{m!}\E \left(\int_0^\tau \gamma(X_s)ds \right)^m &=\frac1{m!} \int_0^\infty e^{-s} \E \left(\int_0^{s} \gamma(X_u)du \right)^m ds \\
&=\int_0^\infty e^{-s} \int_{[0,s]_<^m}   \E\left[\prod_{k=1}^m \gamma(X_{s_{k}})  \right] ds_1\dots ds_m ds\\
&=\int_{\R_{+,<}^m}e^{-s_m} \E\left[\prod_{k=1}^m \gamma(X_{s_{k}})  \right] ds_1\dots ds_m.
\end{align*}
Therefore, 
\begin{align}
 &\frac1{(m+n)!}\E \left(\int_0^\tau \gamma(X_s)ds \right)^{m+n} \notag\\
 =& \ \int_{\R_{+,<}^{m+n}}e^{-s_{m+n}} \E\left[\prod_{k=1}^{m+n} \gamma(X_{s_{k}})  \right] ds_1\dots ds_{m+n}\notag \\
 =&\ \int_{\R_{+,<}^{m+n}} e^{-s_{m}} e^{-(s_{m+n}-s_m)} \E \left[\prod_{k=1}^{m}  \gamma(X_{s_{k}})\E\left[\prod_{k=m+1}^{m+n}  \gamma(X_{s_{k}})\Big| \mathcal F_{s_m}\right]  \right] ds_1\dots ds_{m+n}.\label{eq4.4}
 \end{align}
 For $k=m+1,\dots, m+n$, let $X_{s_k}=\int_0^{s_{k}} K(s_k-s) dW_s= A_{s_m, s_k}+ Y_{s_m, s_k}$, where $A_{s_m,s_k}=\int_0^{s_m} K(s_k-s)dW_s$ and $Y_{s_m, s_k}=\int_{s_m}^{s_k} K(s_k-s)dW_s$. Furthermore, note that 
 \begin{align}\label{eq-distr}
 (Y_{s_m, s_{m+1}}, \dots, Y_{s_m, s_{m+n}})\overset{d}= (X_{s_{m+1}-s_m}, \dots, X_{s_{m+n}-s_m}).
 \end{align}  Hence, by the fact that  $A_{s_m,s_k}\in \mathcal F_{s_m}$, \eqref{eq-distr} and Remark \ref{rem-max-principle}, we have
 \begin{align*}
 \E\left[\prod_{k=m+1}^{m+n}  \gamma(X_{s_{k}})\Big| \mathcal F_{s_m}\right] \le \E\left[\prod_{k=m+1}^{m+n}  \gamma(X_{s_k-s_m})\right].
 \end{align*}
 Thus, it follows from \eqref{eq4.4} and a change of variables,
 \begin{align*}
 &\frac1{(m+n)!}\E \left(\int_0^\tau \gamma(X_s)ds \right)^{m+n}\\
 \le&\ \int_{\R_{+,<}^{m+n}} e^{-s_{m}} e^{-(s_{m+n}-s_m)} \E \left[\prod_{k=1}^{m}  \gamma(X_{s_{k}})\right]  \E\left[\prod_{k=m+1}^{m+n}  \gamma(X_{s_k-s_m})\right] ds_1\dots ds_{m+n}\\
 \le&\ \int_{\R_{+,<}^{m}} e^{-s_{m}}  \E \left[\prod_{k=1}^{m}  \gamma(X_{s_{k}})\right]   ds_1\dots ds_{m} \int_{\R_{+,<}^{n}} e^{-s_{n}}\E \left[\prod_{k=1}^{n}  \gamma(X_{s_{k}})\right]ds_1\dots ds_{n}\\
 =&\ \frac1{m!}\E \left(\int_0^\tau \gamma(X_s)ds \right)^{m}\frac1{n!}\E \left(\int_0^\tau \gamma(X_s)ds \right)^{n}.
 \end{align*}
The proof is completed.\hfill
\end{proof}

\begin{proposition}\label{prop-3.2}
Let $\gamma(x)$ be given in Section \ref{section2}. Then for all $\theta>0,$  $$\lim\limits_{t\to\infty}\frac1t\log \E \exp\left(\theta\int_0^t \gamma(X^\alpha_s)ds \right)=\mathcal E(\gamma, \alpha, \beta)\,\theta^{\frac1{1-\alpha\beta}},$$
where $\mathcal E(\gamma, \alpha, \beta)$ is a positive constant depending on $(\gamma, \alpha, \beta)$.
\end{proposition}
\begin{proof}
Let $\tau$ be an exponential time with parameter 1 which is independent of $X$. By Proposition \ref{prop-sub-additive} and Fekete's lemma, we know that there exists an extended number  $A\in [-\infty, \infty)$, such that 
\begin{align}
A:=&\ \lim_{m\to\infty}\frac1m \log\left(\frac1{m!} \E\left(\int_0^\tau \gamma(X^\alpha_s)ds\right)^m\right)\notag\\
 =&\ \inf_m\left\{ \frac1m \log\left(\frac1{m!} \E\left(\int_0^\tau \gamma(X^\alpha_s)ds\right)^m\right)\right\}. \label{eq3.3}
\end{align}
By the scaling property and the independence of $\tau$ and $X$, we have 
\begin{align}\label{eq-scaling}
&\frac1m \log\left(\frac1{m!} \E\left(\int_0^\tau \gamma(X^\alpha_s)ds\right)^m\right)=\frac1m \log\left(\E[\tau^{(1-\alpha\beta)m}] \frac{1}{m!} \E\left(\int_0^1 \gamma(X^\alpha_s)ds\right)^m\right).
\end{align}

First we show that $A$ is a real number. Noting that for $x\in \R^d$, $|x|^{-(\beta_1+\cdots+\beta_d)}\le |x_1|^{-\beta_1}\cdots|x_d|^{-\beta_d}$, we only need to show $A>-\infty$ for the cases $\gamma(x)=\delta(x)$ and $\gamma(x)=|x|^{-\beta}$. 

For the case $\gamma(x)=\delta(x)$, we have
\begin{align}
&\frac{1}{m!} \E\left(\int_0^1 \gamma(X^\alpha_s)ds\right)^m\notag\\
=&\ \int_{[0,1]_<^m} \E\left[\prod_{i=1}^m\gamma(X^\alpha_{s_i})\right]ds_1\dots ds_m\notag\\
=&\ \int_{[0,1]_<^m} \int_{\R^{md}}\prod_{i=1}^m \widehat\gamma(\xi_i) \exp\left(-\frac12 \Var\left(\sum_{i=1}^m \xi_i \cdot X^\alpha_{s_i}\right)\right)d\xi_1\dots d\xi_m ds_1\dots ds_m.\label{eq-m}
\end{align}
Note (see, e.g., \cite{Berman} or \cite[Lemma 3.8]{clrs}) that, for $0< s_1<s_2<\dots<s_n<1$,
\begin{align}\label{eq-cov}
 &\text{det}[\text{Cov}(X^{\alpha,1}_{s_1},\dots, X^{\alpha,1}_{s_m})]\notag\\
 =&\ \Var(X^{\alpha,1}_{s_1})\Var(X^{\alpha,1}_{s_2}|X^{\alpha,1}_{s_1}) \cdots \Var(X^{\alpha,1}_{s_m}|X^{\alpha,1}_{s_1},\dots, X^{\alpha,1}_{s_{m-1}})\notag\\
=&\ \Var(X^{\alpha,1}_{s_1})\Var(X^{\alpha,1}_{s_2}-X^{\alpha,1}_{s_1}|X^{\alpha,1}_{s_1}) \cdots \Var(X^{\alpha,1}_{s_m}-X^{\alpha,1}_{s_{m-1}}|X^{\alpha,1}_{s_1},\dots, X^{\alpha,1}_{s_{m-1}})\notag\\
\le&\  \Var(X^{\alpha,1}_{s_1})\Var(X^{\alpha,1}_{s_2}-X^{\alpha,1}_{s_1}) \cdots \Var(X^{\alpha,1}_{s_m}-X^{\alpha,1}_{s_{m-1}})\notag\\
\leq &\ C^ms_1^{2\alpha}(s_2-s_1)^{2\alpha}\dots (s_m-s_{m-1})^{2\alpha},
 \end{align}
 where $X^{\alpha,1}$ is the first component of the vector $X$ and $C$ here and  in the following denotes a generic positive constant independent of $m$ which may vary from line to line.

When $\gamma(x)=\delta(x)$, $\widehat\gamma(\xi)=1$, and together with \eqref{eq-cov}, the equation \eqref{eq-m} equals
\begin{align*}
&\int_{[0,1]_<^m} \int_{\R^{md}}\exp\left(-\frac12 \Var\left(\sum_{i=1}^m \xi_i\cdot X^\alpha_{s_i}\right)\right)d\xi_1\dots d\xi_m ds_1\dots ds_m\\
=&\  \int_{[0,1]_<^m} \Big(2\pi \text{det}[\text{Cov}(X^{\alpha,1}_{s_1},\dots, X^{\alpha,1}_{s_m})]\Big)^{-d/2} ds_1\dots ds_m\\
\ge&\  C^m \int_{[0,1]_<^m}  s_1 ^{-\alpha d}(s_2-s_1)^{-\alpha d}\cdots (s_{m}-s_{m-1})^{-\alpha d} ds_1\dots ds_m\\
=&\ C^m \frac{\Gamma^m(1-\alpha d)}{\Gamma(1+(1-\alpha d)m)}.
\end{align*}
The above inequality, the fact $\E[\tau^{(1-\alpha d)m}]=\Gamma(1+(1-\alpha d)m)$ with $\Gamma(\cdot)$ being the Gamma function, \eqref{eq3.3} and \eqref{eq-scaling} imply that $A\ge \log \Gamma(1-\alpha d)+\log C>-\infty$.

 Now we show that $A>-\infty$ when $\gamma(x)=|x|^{-\beta}.$  Assume instead that $A=-\infty$, then by \eqref{eq-scaling}, the fact $\E[\tau^{(1-\alpha \beta)m}]=\Gamma(1+(1-\alpha \beta)m)$, and the Stirling formulas $\Gamma(1+x)\sim \sqrt{2\pi x}\left(\frac xe\right)^x$ and $m!\sim \sqrt{2\pi m}\left(\frac me\right)^m$,  we have
 \begin{equation}\label{eq-4.9}
  \lim_{m\to\infty}\frac1m \log\left(\frac1{(m!)^{\alpha \beta}} \E\left(\int_0^1 \gamma(X^\alpha_s)ds\right)^m\right)=-\infty.
  \end{equation}
 Thus Lemma \ref{lemma10} (i) implies that 
 \begin{equation}\label{song-10-14}
 \limsup_{u\to\infty}\frac1{u^{1/\alpha\beta}}\log \mathbb P\left(\int_0^1 \gamma (X^\alpha_s)ds\ge u\right)=-\infty.
 \end{equation}
For any $\lambda>0$, by the scaling property $\int_0^t \gamma(X^\alpha_s)ds\overset d= t^{1-\alpha\beta}\int_0^1 \gamma(X^\alpha_s)ds$, and  by a change of variables $u=\lambda t^{\alpha\beta}$, we have
 \begin{align}\label{eq-4.10}
 \limsup_{t\to\infty}\frac1t \log \mathbb P\left( \frac 1t \int_0^t \gamma (X^\alpha_s)ds\ge \lambda \right)
 =& \limsup_{t\to\infty}\frac1t \log \mathbb P\left( \int_0^1 \gamma (X^\alpha_s)ds\ge \lambda t^{\alpha\beta} \right)\notag\\
 =& \limsup_{u\to\infty}\frac{\lambda^{1/\alpha\beta}}{u^{1/\alpha\beta}}\log \mathbb P\left(\int_0^1 \gamma (X^\alpha_s)ds\ge u\right)=-\infty.
 \end{align}
Now take $I(\lambda)=\lambda$ for all $\lambda\geq 0$. Then, $I(\lambda)$ is a non-decreasing rate function on $\mathbb{R}^+$ with $I(0)=0$, and \eqref{eq-4.10} yields
\begin{equation}\label{song-3-10}
\limsup_{t\to\infty}\frac1t \log \mathbb P\left( \frac 1t \int_0^t \gamma (X^\alpha_s)ds\ge \lambda \right)\leq -I(\lambda), \mbox{ for all } \lambda>0.
\end{equation}
 Note  that \eqref{song-10-14} and Lemma \ref{lem-3-3} with $p=1/\alpha\beta$ imply 
 $$\E\exp\left(\rho\int_0^1 \gamma(X^\alpha_s)ds\right)<\infty,\ \mbox{ for all }\rho>0,$$
 and 
 \[
 \limsup\limits_{\rho\to \infty} \rho^{-\frac{1}{1-\alpha\beta}}\log  \mathbb{E}\, \exp\left(\rho\int_0^1 \gamma(X^\alpha_s)ds\right)<\infty.
 \]
Hence, by the scaling property $\int_0^t \gamma(X^\alpha_s)ds\overset d= t^{1-\alpha\beta}\int_0^1 \gamma(X^\alpha_s)ds$ and a change of variables $\rho=\theta t^{1-\alpha\beta}$ for any fixed $\theta>0$, we can obtain
 \begin{align}\label{eq-uni-exp}
& \limsup\limits_{t\to\infty}\frac1t\log \E\exp\left(\theta\int_0^t \gamma(X^\alpha_s)ds\right)\notag\\
=&\ \limsup\limits_{t\to\infty}\frac1t\log \E\exp\left(\theta t^{1-\alpha\beta}\int_0^1 \gamma(X^\alpha_s)ds\right)\notag\\
=&\ \theta^{\frac{1}{1-\alpha\beta}} \limsup\limits_{\rho\to\infty} \rho^{-\frac{1}{1-\alpha\beta}}\log  \mathbb{E}\, \exp\left(\rho\int_0^1 \gamma(X^\alpha_s)ds\right)<\infty.
 \end{align}
Now, \eqref{eq-uni-exp} implies that (1.2.29) with $p=1$ in \cite[ Lemma 1.2.10]{chen} holds. Then, Lemma 1.2.10 in \cite{chen} implies that the assumption (1.2.26) in \cite{chen} holds. Together with \eqref{song-3-10}, we apply Theorem 1.2.9(2) with $p=1$ in \cite{chen} to obtain
 \begin{align*}
 \limsup_{t\to\infty}\frac1t \log \E\exp\left(\int_0^t \gamma(X^\alpha_s)ds\right)=&\ \limsup_{t\to\infty}\frac1t \log \left\{\sum_{m=0}^\infty\frac{1}{m!}\E\left(\int_0^t \gamma(X^\alpha_s)ds\right)^m\right\}\\
 \leq&\ \sup_{\lambda>0}\left\{\lambda-I(\lambda)\right\}=\sup_{\lambda>0}\{0\}=0.
 \end{align*}
 which contradicts  Lemma \ref{lemma8}. Therefore, $A>-\infty$ when $\gamma(x)=|x|^{-\beta}$.

 Since $A$ in \eqref{eq3.3} is a real number,  by \eqref{eq-scaling} and the Stirling formula, there exists $a\in (-\infty, \infty)$ (depending on the function $\gamma(x)$) such that 
  \begin{equation*}
  \lim_{m\to\infty}\frac1m \log\left(\frac1{(m!)^{\alpha \beta}} \E\left(\int_0^1 \gamma(X^\alpha_s)ds\right)^m\right)=a.
  \end{equation*}
  It implies from Lemma \ref{lemma10} (ii) that
  \begin{equation}\label{eq-ab}
  \lim\limits_{u\to\infty} \frac1{u^{1/\alpha\beta}}\log \mathbb P\left(\int_0^1 \gamma (X^\alpha_s)ds\ge u\right)=-\alpha\beta e^{-a/\alpha\beta}<0.
  \end{equation}
Hence, by  the scaling property $\int_0^t \gamma(X^\alpha_s)ds\overset d= t^{1-\alpha\beta}\int_0^1 \gamma(X^\alpha_s)ds$ and a change of variables, we have for all $\lambda>0$,
\begin{align*}
 \lim_{t\to\infty} \frac1t \log \mathbb P\left(\frac1t\int_0^t \gamma(X^\alpha_s)ds \ge \lambda \right)=&\ \lim_{t\to\infty} \frac1t \log \mathbb P\left(\frac1t \cdot t^{1-\alpha\beta}\int_0^1 \gamma(X^\alpha_s)ds \ge \lambda \right)\\
 =&\ \lim_{t\to\infty} \frac1t \log \mathbb P\left(\int_0^1 \gamma(X^\alpha_s)ds \ge \lambda t^{\alpha\beta} \right)\\
 =&\ \lambda^{1/\alpha\beta}\lim_{u\to\infty} \frac1{u^{1/\alpha\beta}}\log \mathbb P\left(\int_0^1 \gamma(X^\alpha_s)ds \ge u \right)\\
=&\  -\alpha\beta e^{-a/\alpha\beta} \lambda^{1/\alpha\beta}.
 \end{align*}
Notice that the scaling property $\int_0^t \gamma(X^\alpha_s)ds\overset d= t^{1-\alpha\beta}\int_0^1 \gamma(X^\alpha_s)ds$,  a change of variables, \eqref{eq-ab} and Lemma \ref{lem-3-3} with $p=1/\alpha\beta$ imply that, for all $\theta>0$, 
\begin{align*}
&\limsup\limits_{t\to\infty}\frac1t\log \E\exp\left(\theta\int_0^t \gamma(X^\alpha_s)ds\right)\\
=&\ \limsup\limits_{t\to\infty}\frac1t\log \E\exp\left(\theta t^{1-\alpha\beta}\int_0^1\gamma(X^\alpha_s)ds\right)\\
=&\ \theta^{\frac1{1-\alpha\beta}} \limsup\limits_{\rho\to\infty}\rho^{-\frac{1}{1-\alpha\beta}}\log\E\exp\left(\rho\int_0^1\gamma(X^\alpha_s)ds\right)<\infty.
\end{align*}
Hence, by Theorem \ref{equiv} and Varadhan's integral lemma (see Lemma \ref{lem-vil}), we have for $\theta>0$,
$$ \lim_{t\to\infty} \frac1t \log \E\exp\left(\theta\int_0^t \gamma(X^\alpha_s)ds \right)= \mathcal E(\gamma, \alpha, \beta) \,\theta^{\frac1{1-\alpha\beta}}.$$
The proof is concluded.\hfill
\end{proof}

The remaining part of this section consists of lemmas that were used in the previous proof. 
 
\begin{lemma}\label{lem-3-3}
Let $Y$ be a nonnegative random variable, and let $b\in(0,\infty]$ and $p>1$ be given. If $\lim\limits_{t\to\infty}\frac{1}{t^p}\log \mathbb{P}(Y\geq t)=-b,$ then
\begin{itemize}
\item[(a)] $\mathbb{E}\, e^{\rho Y}<\infty, \ \mbox{ for any  } \rho> 0;$
\item[(b)] $\limsup\limits_{\rho\to \infty} \rho^{-\frac{p}{p-1}}\log  \mathbb{E}\, e^{\rho Y}<\infty$.
\end{itemize}
\end{lemma}
\begin{proof} Note that the sequence $\{\frac{1}{t^p}\log \mathbb{P}(Y\geq t)\}$ is decreasing as $t\uparrow \infty$. Since $\lim\limits_{t\to\infty}\frac{1}{t^p}\log \mathbb{P}(Y\geq t)=-b$, by setting $B$ to be any positive number in the case $b=\infty$ and setting $B=b/2$ in the case $b<\infty$, there exists $T>0$ such that for any $t\geq T$
\[
\mathbb{P}(Y\geq t)\leq e^{-B t^p}, 
\]
and hence, by Fubini's theorem
\begin{align}\label{exp1}
\mathbb{E}\, e^{\rho Y}=&\ 1+\int_0^\infty \rho e^{\rho t}P(Y\geq t)dt\notag\\
=&\ 1+\int_0^T \rho e^{\rho t}P(Y\geq t)dt+\int_T^\infty \rho e^{\rho t}P(Y\geq t)dt\notag\\
\leq &\ e^{\rho T}+\int_{T}^\infty\rho e^{\rho t-B t^p}dt\notag\\
=&\ e^{\rho T}\left(1+\int_{T}^\infty \rho e^{\rho(t-T)-Bt^p}dt\right)\notag\\
\leq &\ e^{\rho T}\left(1+\int_{T}^\infty \rho e^{\rho(t-T)-B(t-T)^p}dt\right)\notag\\
=&\  e^{\rho T}\left(1+\int_{0}^\infty \rho e^{\rho t- Bt^p}dt\right)<\infty.
\end{align}
Note that
\begin{align}\label{ineq 1}
 \int_{0}^\infty \rho e^{\rho t-Bt^p}dt
&\geq \ \int_{0}^{\left(\frac\rho{B}\right)^{\frac{1}{p-1}}} \rho e^{\rho t\left(1-\frac{B}{\rho}t^{p-1}\right)}dt
\geq \int_{0}^{\left(\frac\rho{B}\right)^{\frac{1}{p-1}}} \rho dt= \rho^{\frac{p}{p-1}}B^{-\frac{1}{p-1}}\to \infty,
\end{align}
as $\rho\to\infty$. 
Note also that
\begin{align}\label{ineq2}
& \int_{0}^\infty \rho e^{\rho t-Bt^p}dt\notag\\
=&\ \int_{0}^{\left(\frac\rho{B}\right)^{\frac{1}{p-1}}} \rho e^{\rho t\left(1-\frac{B}{\rho}t^{p-1}\right)}dt+\int_{\left(\frac\rho{B}\right)^{\frac{1}{p-1}}}^{\left(\frac{2\rho}{B}\right)^{\frac{1}{p-1}}} \rho e^{-\rho t\left(\frac{B}{\rho}t^{p-1}-1\right)}dt+\int_{\left(\frac{2\rho}{B}\right)^{\frac{1}{p-1}}}^\infty \rho e^{-\rho t\left(\frac{B}{\rho}t^{p-1}-1\right)}dt\notag\\
\leq &\ \int_{0}^{\left(\frac\rho{B}\right)^{\frac{1}{p-1}}} \rho e^{\rho t}dt+\int_{\left(\frac\rho{B}\right)^{\frac{1}{p-1}}}^{\left(\frac{2\rho}{B}\right)^{\frac{1}{p-1}}} \rho dt+\int_{\left(\frac{2\rho}{B}\right)^{\frac{1}{p-1}}}^\infty \rho e^{-\rho t}dt\notag\\
=&\ \exp\left\{\rho^{\frac{p}{p-1}}B^{-\frac{1}{p-1}}\right\}-1+\rho^{\frac{p}{p-1}}B^{-\frac{1}{p-1}}\left(2^{\frac1{p-1}}-1\right)+\exp\left\{-\rho^{\frac{p}{p-1}}\left(\frac{2}{B}\right)^{\frac{1}{p-1}}\right\}\notag\\
\leq&\  \exp\{\rho^{\frac{p}{p-1}}B^{-\frac{1}{p-1}}\}+\rho^{\frac{p}{p-1}}B^{-\frac{1}{p-1}}\left(2^{\frac1{p-1}}-1\right).
\end{align}
It is easy to see that 
\begin{equation}\label{eq1}
\rho^{\frac{p}{p-1}}B^{-\frac{1}{p-1}}\left(2^{\frac1{p-1}}-1\right)\leq \exp\{\rho^{\frac{p}{p-1}}B^{-\frac{1}{p-1}}\}
\end{equation}
if $\rho$ is large enough.

Using the fact $\log(1+x)\leq 1+\log x$ for all $x\geq 1$, by \eqref{exp1}-\eqref{eq1}, when $\rho $ is large enough, we have
\begin{align*}
\rho^{-\frac{p}{p-1}}\log  \mathbb{E}\, e^{\rho Y}&\leq T\rho^{-\frac{1}{p-1}}+\rho^{-\frac{p}{p-1}}\left(1+\log \int_{0}^\infty \rho e^{\rho t-Bt^p}dt\right)\\
&\leq T\rho^{-\frac{1}{p-1}}+\rho^{-\frac{p}{p-1}}\left(1+\log 2+\rho^{\frac{p}{p-1}}B^{-\frac{1}{p-1}}\right).
\end{align*}
Therefore, we can show that
\[
\limsup\limits_{\rho\to \infty} \rho^{-\frac{p}{p-1}}\log  \mathbb{E}\, e^{\rho Y}\leq B^{-\frac{1}{p-1}}<\infty.
\]
The proof is completed.
\end{proof}

\begin{lemma}\label{lemma8} 
For the Riemann-Liouville process $X^\alpha$ we have
$$\liminf_{t\to\infty} \frac1t\log \E \exp\left(\int_0^t |X^\alpha_s|^{-\beta} ds \right)>0.$$ 
\end{lemma}
\begin{proof}
For any $\e >0$, by the small ball probability result provided in  \cite[Theorem 4.1]{ls}, there exists a constant $c_0\in (0, \infty)$ such that
\begin{equation}\label{eq-sb}
\mathbb P\left(\sup_{0\le s\le 1} |X_s^{\alpha,1}| \le \e\right) 
\ge \exp(-c_0 \e^{-1/\alpha}).
\end{equation}
Denote $S_\e=\left\{\sup\limits_{j\in\{1,\dots, d\}} \sup\limits_{0\le s\le 1}|X_s^{\alpha,j}| \le \e\right\}$. By \eqref{eq-sb}, we have $$P(S_\e)= \left(\mathbb P\left(\sup\limits_{0\le s\le 1} |X_s^{\alpha,1}| \le \e\right)\right)^d\ge\exp(-c_0d \e^{-1/\alpha}).$$ 
Then,  
\begin{align*}
\E \exp\left(\int_0^t |X^\alpha_s|^{-\beta} ds \right) \ge&\ \E \left[\exp\left(t^{1-\alpha\beta}\int_0^1 |X^\alpha_s|^{-\beta} ds \right)\mathbf1_{S_\e}\right]\\
\ge&\ \exp\left(c_dt^{1-\alpha\beta}\e^{-\beta}\right) \mathbb P(S_\e)\\
\ge&\ \exp(c_dt^{1-\alpha\beta}\e^{-\beta}-c_0 d \e^{-1/\alpha} ),
\end{align*}
where $c_d$ is a positive constant depending on 
$d$. 

Now, choose $\e=\left(\frac{2c_0d}{ c_d}\right)^{\alpha/(1-\alpha\beta)} t^{-\alpha}$ such that \[
c_dt^{1-\alpha\beta}\e^{-\beta}-c_0 d \e^{-1/\alpha}=c_0d\e^{-1/\alpha}= C t,
\]
where $C=c_0d\left(\frac{c_d}{2c_0d}\right)^{1/(1-\alpha\beta)}$. Consequently, we can show
$$\liminf_{t\to\infty}\frac1t\log \E \exp\left(\int_0^t |X^\alpha_s|^{-\beta} ds\right)\ge C>0.$$
We prove the desired result.\hfill
\end{proof}

The following lemma (see \cite[Lemma 2.3]{km02}) connects the moments and large deviations. 
\begin{lemma}\label{lemma10}
Let $F\ge 0$ be a random variable and let   $p >0$. Then, for any $a\in \R$, the following results hold. 
\begin{itemize} 
\item[(i)] If $$\limsup_{m\to \infty} \frac1m\log\left(\frac1{(m!)^p}\E F^m\right)\le a,$$
for some $a\in \R$, then
$$ \limsup_{x\to\infty} \frac1{x^{1/p}} \log \mathbb P(F\ge x)
\le -p e^{-a/p}.$$

\item[(ii)] If  $$\lim_{m\to \infty} \frac1m\log\left(\frac1{(m!)^p}\E F^m\right)=a,$$
for some $a\in \R$, then
$$ \lim_{x\to\infty} \frac1{x^{1/p}} \log \mathbb P(F\ge x)=-p e^{-a/p}.$$
\end{itemize}

\end{lemma}

\section{Large deviations for the functionals of fBm, sub-fBm and bi-fBm}\label{section4}

\subsection{Some preliminaries on fBm,  sub-fBm and bi-fBm}\label{section4.1}

In this subsection, we will first recall some preliminaries on fBm,  sub-fBm and bi-fBm, and then we will provide some detailed results on the RKHSs associated to some of them. 

\begin{definition}
A centered $1$-dimensional Gaussian process $\{B_t^H, t\ge 0\}$ is called a fBm  with Hurst parameter $H\in (0,1),$ if the covariance function is given by
\begin{equation}\label{cov-fbm}
R(t,s)=\E[B_t^H B_s^H]=\frac{1}{2}\left(t^{2H}+s^{2H}-|t-s|^{2H}\right).
\end{equation}
\end{definition}

It immediately implies from the above covariance function that fBm has self-similarity: $\{B^H_{at}, t\geq 0\}\overset{d}{=}a^H\{B_t^H, t\geq 0\}$, for any $a>0$. 

When $H=\frac12$, the process $B^{\frac12}$ is a standard Brownian motion. As an extension of classical Brownian motion,  fBm is essentially different  from Brownian motion in the sense that fBm is not a semi-martingale nor a Markov process when $H\neq \frac12$. We refer readers to \cite{bhoz, nualart06} and the references therein for more details on the analysis of fBm. 

In this subsection, let $W=\{W_t, t\in \R\}$ be a Brownian motion on $\R$. Then, fBm $B^H$ has the following representation (see \cite{M-VN} and \cite{st}):
\begin{align}\label{dec-fbm}
B^H_t=&\ \alpha_H\int_{-\infty}^t\left[(t-s)^{H-\frac12}-(-s)^{H-\frac12}_+\right]dW_s\notag\\
=&\ \alpha_H X^{H}_t+\alpha_H\int_{-\infty}^0\left[(t-s)^{H-\frac12}-(-s)^{H-\frac12}_+\right]dW_s\notag\\
=&: \alpha_H X^{H}_t+\eta^H_t, 
\end{align}
where $\alpha_H=\left(\int_0^\infty \left[(1+s)^{H-\frac12}-s^{H-\frac12}\right]^2ds+\frac1{2H}\right)^{\frac12}$, $X^H$ is the Riemann-Liouville process with parameter $H$, and the process $\eta^H_t= \alpha_H\int_{-\infty}^0\left[(t-s)^{H-\frac12}-(-s)^{H-\frac12}_+\right]dW_s$ is independent of $X^H$.

The left-sided fractional Riemann-Liouville
integrals of $f\in L^2[0,1]$ of order $\alpha>0 $ are defined for almost all $t\in \left[
0,1\right] $ by 
\[
I_{0^+}^\alpha f(t) =\frac{1}{\Gamma(\alpha)} \int_0^t (t-s)^{\alpha-1}f(s)ds.
\]
Let $I_{0+}^{\alpha}(L^2[0,1])$ denote the image of $L^2[0,1]$ under $I^\alpha_{0^+}$.

Lemma 10.2 in \cite{vdv-vz} provides the following result on the RKHS of the Riemann-Liouville process $X^H$:
\[
\bh(X^H)=I^{H+\frac12}_{0+}(L^2[0,1]).
\]

Regarding the decomposition \eqref{dec-fbm} of fBm $B^H$, we summarize Propositions 3.3 and 3.5 in \cite{clrs} as follows.
\begin{proposition}\label{RKHS-RV}
For any $0<\varepsilon<1$,  the process $\{\eta^H_t, t\geq \varepsilon\}$ has $C^\infty$-sample paths a.s., and 
 there is a Gaussian process $\{\eta_{t}^{\varepsilon,H}, t\geq 0 \}$ such that
\begin{itemize}
\item[(a)] $\eta_{t}^{\varepsilon,H}=\eta^H_t$ for all $t\geq \varepsilon$;
\item[(b)] 
$
\mathbb{P}\left(\left\{\eta^{\varepsilon,H}_{t}, 0\leq t\leq 1\right\}\in \bh(X^H)\right)=1.
$
\end{itemize}
\end{proposition}

\bigskip
Next, we discuss the RKHS of fBm $B^H$. The covariance function of fBm can be expressed as (see \cite{du}) 
\begin{equation}\label{cov-fbm-KH}
R(t,s)=\E[B_t^HB_s^H]=\int_0^{t\wedge s}K_H(t,r)K_H(s,r)dr, 
\end{equation}
where $$K_H(t,s)=c_H(t-s)^{H-\frac12}F\left(H-\frac12, \frac12-H, H+\frac12, 1-\frac ts\right)\mathbf{1}_{[0,t]}(s),$$
with $c_H=\left[\frac{2H\Gamma\left(\frac{3}{2}-H\right)}{\Gamma(2-2H)\Gamma\left(H+\frac12\right)}\right]^{1/2}$ and $F(a,b, c,z)$ being the Gauss hypergeometric function.

Note that for any fixed $t\in [0,1]$, $K_H(t,\cdot)\in L^2[0,1]$. Consider the integral transform $\mathcal{K}_H$ defined on $L^2[0,T]$ by 
$$ (\mathcal{K}_H f)(t):= \int_0^t K_H(t,s) f(s) ds, \mbox{ for any } f\in L^2[0,1].$$ 

We can easily see from \eqref{cov-fbm-KH} that  $R(s,\cdot)=\mathcal{K}_H(K_H(s,\cdot))$. In fact, the RKHS of $B^H$ is $\mathbb{H}(B^H)=I_{0+}^{H+\frac12}(L^2[0,1])$ endowed with the inner product $\langle\cdot, \cdot\rangle_{\mathbb{H}(B^H)}$. and the integral transform $\mathcal{K}_H$ is an isomorphism from $L^2[0,1]$ onto  $I_{0+}^{H+\frac12}(L^2[0,1])$ (see \cite[Theorems 2.1 and 3.3 and Remark 3.1]{du}).

\bigskip
 \begin{definition}
 Let $\{S_t^H, t\geq 0\}$ denote a $1$-dimensional  sub-fBm with index $H\in (0,1)$ introduced in \cite{BGT}, which is a centered Gaussian process with covariance function 
\begin{equation}\label{cov-subfbm}
\E[S_t^H S_s^H]=t^{2H}+s^{2H}-\frac12\left((t+s)^{2H}+|t-s|^{2H}\right).
\end{equation}
\end{definition}

Clearly when $H=\frac12$, $S^{\frac12}$ is a standard Brownian motion. Note that the process $S^H$  has self-similarity $\{S^{H}_{at}, t\geq 0\}\overset{d}{=}a^{H}\{S_t^{H}, t\geq 0\}$, for any $a>0$.

 Define, for $\alpha\in \left(0,\frac12\right)$,
\begin{equation}\label{eq-y}
Y_t^\alpha= \int_0^\infty (1-e^{-r t}) r^{-\alpha-\frac{1}{2}} dW_r.
\end{equation}
By some basic calculation, we get
\begin{equation}\label{cov-Y}
\E(Y^\alpha_tY_s^\alpha)=\frac{\Gamma(1-2\alpha)}{2\alpha}\left(t^{2\alpha}+s^{2\alpha}-(t+s)^{2\alpha}\right).
\end{equation}

Assume that $W$ in \eqref{eq-y} is independent of $B^H$,  part (a) of  \cite[Theorem 3.5]{ct} provides the following decomposition result for sub-fBm. 

\begin{proposition}\label{decom-sub}
For  $0<H<\frac12$, $\left\{B^H_t+\sqrt{\frac{H(1-2H)}{\Gamma(2-2H)}}\, Y^{H}_t, t\geq 0\right\}$ has the same law as $\{S^H_t,\ t\geq 0\}$.
 \end{proposition}

 \begin{definition}
The  bi-fBm $\{Z_t^{H,K}, t\ge 0\}$ with parameters $H\in (0,1)$ and $K\in (0,1]$ is a generalization of fBm, defined as  a centered $1$-dimensional Gaussian process with covariance function 
\begin{equation}\label{cov-bifbm}
\E[Z_t^{H,K}Z_s^{H,K}]=\frac1{2^K}\Big( (t^{2H}+s^{2H})^K -|t-s|^{2HK} \Big). 
\end{equation}
\end{definition}

It is obvious that the bi-fBm $Z^{H,1}$ degenerates to a fBm when $K=1$ and that $Z^{H,K}$ has self-similarity $\{Z^{H,K}_{at}, t\geq 0\}\overset{d}{=}a^{HK}\{Z_t^{H,K}, t\geq 0\}$, for any $a>0$.

Now assume that $W$ in \eqref{eq-y} is independent of $Z^{H,K}$. The following decomposition result for bi-fBm is obtained in \cite{ln}.
\begin{proposition}\label{decom-bi}
For $H, K\in(0,1)$, $\left\{Z_t^{H,K}+ \sqrt{\frac{K}{2^{K}\Gamma(1-K)}}~Y_{t^{2H}}^{\frac{K}{2}}, t\geq 0\right\}$ has the same law as $\left\{2^{\frac{1-K}{2}}B^{HK}_t, t\geq 0\right\}$, where  $B^{HK}$ is a fBm with Hurst parameter $HK$. 
\end{proposition}

 Denote the RKHSs of $\left\{B^H_t, 0\leq t\leq 1\right\}$ with parameter $H$ and $\left\{Z^{H,K}_t, 0\leq t\leq 1\right\}$ with parameters $H$ and  $K$ by $\bh(B^H)$ and $\bh(Z^{H,K})$ respectively. For bi-fBm $Z^{H,K}$, unlike fBm, we don't have an explicit representation for its RKHS when $K\in (0,1)$. However, we have the following property.
 
 \begin{proposition}\label{prop-bi}
For $H, K\in (0,1),$ $\bh(B^H)\subseteq \bh(Z^{H,K})\subseteq \bh(B^{HK}) $. In particular, \\ $I_{0+}^{\frac32}(L^2[0,1])\subseteq \bh(B^H)\subseteq\bh(Z^{H,K}).$    
\end{proposition}
\begin{proof}
The fact $\bh(B^H)\subseteq \bh(Z^{H,K}) $ follows from \cite[Theorem 5.2]{al}.

Next, we will show $\bh(Z^{H,K})\subseteq \bh(B^{HK})$. Note from \eqref{cov-fbm}, \eqref{cov-Y} and \eqref{cov-bifbm} that 
\begin{align*}
\E\left(B_t^{2HK}B_s^{2HK}\right)=&\frac{2^K}{2}\E\left(Z_t^{H,K}Z_s^{H,K}\right)+\frac{K}{2\Gamma(1-K)}\E\left(Y^K_{t^{2H}}Y_{s^{2H}}^K\right).
\end{align*} 
It follows from Theorem I in \cite[p. 354]{a} that the $\mathbb{H}\left(\sqrt{\frac{2^K}{2}}Z^{H,K}\right)=\mathbb{H}(Z^{H,K})\subseteq \mathbb{H}(B^{HK})$.

For any $0\le \alpha\le \beta$, we have $I_{0+}^{\beta}(L^2[0,1])\subseteq I_{0+}^{\alpha}(L^2[0,1])$. Indeed, for any $f \in I_{0+}^{\beta}(L^2[0,1])$, there exists $g\in L^2[0,1]$ such that $f=I_{0+}^{\beta}(g)$. Then by \cite[Theorem 2.5]{skm}, we get $f=I_{0+}^{\alpha}(I_{0+}^{\beta-\alpha}(g))$. In addition, $I_{0+}^{\beta-\alpha}(g)\in L^2[0,1]$ by \cite[Theorem 2.6]{skm}, we can show that $f=I_{0+}^{\beta}(g)=I_{0+}^{\alpha}(I_{0+}^{\beta-\alpha}(g))$ belongs to $I_{0+}^{\alpha}(L^2[0,1])$. Consequently, 
the relationship $I_{0+}^{\frac32}(L^2[0,1])\subseteq \bh(B^H)$ holds by noting $\bh(B^H)=I_{0+}^{H+\frac12}(L^2[0,1])$. 
  \hfill 
\end{proof}

\bigskip

\subsection{Large deviation results}

Throughout this subsection, let $B^H$, $S^H$ and $Z^{H,K}$ be $d$-dimensional processes, and for a general $d$-dimensional Gaussian process $X$, denote 
\[
\L^\gamma_t(X)=\int_0^t \gamma(X_s)ds,
\] where $\gamma(x)$ is given in Section \ref{section2}. In particular, when $\gamma=\delta$, $Hd<1$ and $HKd<1$, from Remarks \ref{rmk-fbm-sub-bi} and \ref{rmk-RL}, the local times $L_t(B^H)=\L_t^\delta(B^H)$, $L_t(S^H)=\L_t^\delta(S^H)$ and $L_t(Z^{H,K})=\L_t^\delta(Z^{H,K})$ exist.

 Due to the self-similarity possessed by $B^H$, $S^H$, $Z^{H,K}$ and the homogeneity  $\gamma(ax)=a^{-\beta}\gamma(x)$ for $a>0$,   the following scaling property holds: for any $a>0$,
\begin{align}
 \L_{at}^\gamma(B^H) \overset{d}{=}&\ a^{1-H\beta} \L_t^\gamma(B^H),\label{scaling-1}\\
\L_{at}^\gamma(S^H) \overset{d}{=}&\ a^{1-H\beta} \L_t^\gamma(S^H), \label{scaling-2}\\
  \L_{at}^\gamma(Z^{H,K}) \overset{d}{=}&\ a^{1-HK\beta} \L_t^\gamma(Z^{H,K}).\label{scaling-3}
\end{align}

We first provide the following large deviation result for $\L_t^\gamma(B^H)$, which will be used later to obtain the large deviation for $\L_t^\gamma(S^H)$ with $H\in(0,\frac12)$ and $Hd<1$,  and $\L_t^\gamma(Z^{H,K}) $ with $H\in(0,1)$, $K\in(0,1)$ and $HKd<1$.

\begin{theorem}\label{thm-fbm}
Assume $H\in(0,1)$ and $H\beta<1,$ then
\begin{equation}\label{ldp-fbm-exp}
\lim_{t\to\infty}\frac1t\log \E\exp\left(\theta \L_t^\gamma(B^H)  \right)=\mathcal E(\gamma, H,\beta)\, \alpha_H^{-\frac{\beta}{1-H\beta}}\, \theta^{\frac{1}{1-H\beta}}=:\mathcal{E}_1(\gamma, H,\beta)\, \theta^{\frac{1}{1-H\beta}}, 
\end{equation}
and consequently, 
\begin{equation}\label{ld-p-1}
\lim_{x\to\infty} \frac1{x^{\frac1{H\beta}}} \log \mathbb P(\L^\gamma_1(B^H)\ge x)=-C(\gamma,H,\beta), 
\end{equation}
where $\mathcal E(\gamma, H,\beta)$ is a positive constant given in Proposition \ref{prop-3.2}, the constant $\alpha_H$ is given in \eqref{dec-fbm},  $\mathcal{E}_1(\gamma, H,\beta)=\mathcal E(\gamma, H,\beta)\, \alpha_H^{-\frac{\beta}{1-H\beta}}$, and 
\begin{equation}\label{eq-C}
C(\gamma, \alpha,\beta)= \mathcal E_1(\gamma, \alpha, \beta)^{1-\frac{1}{H\beta}} (1-\alpha\beta)^{\frac1{\alpha\beta}-1}\alpha\beta.
\end{equation}

\end{theorem}
\begin{proof}
By the representation \eqref{dec-fbm} and Proposition \ref{RKHS-RV},   for any $\e\in(0,1)$, we may construct $\eta^\e$  such that $\eta^\e_t=\eta_t$ for $t\ge \e$, and  $\{\eta^\e_t, t\in[0,1]\}$ belongs to the RKHS $\left(I^{H+1/2}_{0+}(L^2[0,1])\right)^{d}$ of $\Big\{\alpha_HX^H_t, t\in[0,1]\Big\}$ almost surely. Thus, \eqref{ldp-fbm-exp} follows from Proposition \ref{prop-ld} and Proposition \ref{prop-3.2}. 

By the G\"artner-Ellis theorem (see Theorem \ref{thm-GE}), we have for $\lambda>0$
\begin{align}\label{sH-1}
\lim_{t\to\infty}\frac1t \log\mathbb P\left(\frac1{t}\, \L^\gamma_t(B^H) \ge \lambda\right)&=\sup_{\theta>0} \left\{ \theta \lambda- \mathcal E_1(\gamma, H, \beta) \theta^{\frac1{1-H\beta}}\right\}\notag\\
&= C(\gamma, H, \beta) \lambda^{\frac1{H\beta}},
\end{align}
where $C(\gamma, H, \beta)$ is given in \eqref{eq-C}. By scaling property \eqref{scaling-1} and a change of variables, we have
\begin{align}\label{sH-2}
\lim_{t\to\infty}\frac1t \log\mathbb P\left(\frac1{t}\, \L^\gamma_t(B^H) \ge \lambda\right)=&\ \lim_{t\to\infty}\frac1t \log\mathbb P\left(\frac1{t^{H\beta}}\, \L^\gamma_1(B^H) \ge \lambda\right)\notag\\
=&\ \lim_{x\to\infty} \frac{\lambda^{\frac1{H\beta}}}{x^{\frac1{H\beta}}} \log \mathbb P(\L^\gamma_1(B^H)\ge x).
\end{align}
Hence, equation \eqref{ld-p-1} follows from \eqref{sH-1} and \eqref{sH-2}.\hfill
\end{proof}

\begin{remark}\label{rem-4.6}
When $\gamma(x)=\delta(x)$ (in this case $\beta=d$), we can prove Theorem \ref{thm-fbm} directly based on the result in \cite[Theorem 2.1]{clrs}. Moreover, we have
 \[
\mathcal E_1(\delta, H,\beta)=C(H,d)^{-\frac{Hd}{1-Hd}}\left[(Hd)^{\frac{Hd}{1-Hd}}-(Hd)^{\frac{1}{1-Hd}}\right],
\]
 where $C(H,d)$ is given by (4.6) in \cite{clrs} and satisfies
\begin{equation}\label{chd}
 \left(\frac{\pi c_H^2}{H}\right)^{\frac1{2H}} \varphi(Hd)\le C(H,d)\le  \left(2\pi\right)^{\frac1{2H}} \varphi(Hd),
\end{equation}
 with $c_H=\frac{\sqrt{2H}2^H}{\sqrt{B(1-H,H+1/2)}}$, $
B(\cdot, \cdot)$ being the beta function and $$\varphi(x)=\frac{x (1-x)^{\frac{1-x}{x}}
}{{\Gamma(1-x)^{\frac1{x}}}}.$$ 

\begin{proof}
It follows from Theorem 2.1 in \cite{clrs} that 
\begin{equation}\label{ld-p'}
\lim_{x\to\infty} \frac1{x^{\frac{1}{Hd}}} \log \mathbb P(L_1(B^H)\ge x)=-C(H,d), \end{equation} 
and $C(H,d)$ satisfies \eqref{chd}.

Fixing a constant $\lambda >0$ and  letting $x=u\lambda$,  we have
$$\lim_{u\to\infty} \frac1{u^{\frac1{Hd}}} \log \mathbb P\left(\frac1u\, L_1(B^H)\ge \lambda\right)=-C(H,d)\lambda^{\frac1{Hd}}. $$
Now, by Varadhan's integral lemma  (see Lemma \ref{lem-vil}),  we get,  for all $\theta>0,$ 
\begin{align*}
\lim_{u\to\infty} \frac1{u^{\frac1{Hd}}}\log \E \exp\left(\theta u^{\frac1{Hd}-1} L_1 (B^H)\right)&=\sup_{\lambda >0}\{\lambda \theta-C(H,d) \lambda^{\frac1{Hd}}\} \\
&=C(H,d)^{-\frac{Hd}{1-Hd}}\left[(Hd)^{\frac{Hd}{1-Hd}}-(Hd)^{\frac{1}{1-Hd}}\right]  \theta^{\frac{1}{1-Hd}}.
\end{align*}
On the other hand, a change of variables and the scaling property \eqref{scaling-1} yield
\begin{align*}
\lim_{u\to\infty} \frac1{u^{\frac1{Hd}}}\log \E \exp\left(\theta u^{\frac1{Hd}-1} L_1 (B^H)\right)=&\ \lim_{t\to\infty}\frac1t\log \E \exp\left(\theta t^{1-Hd} L_1 (B^H)\right)\\
=&\ \lim_{t\to\infty}\frac1t\log \E \exp\left(\theta  L_t (B^H)\right).
\end{align*}
 Therefore, the desired result follows from the above equalities.
\hfill
\end{proof}
\end{remark}

The following result is the large deviations for sub-fBm and bi-fBm. 

\begin{theorem}\label{thm-sfbm}
\begin{itemize}
\item[(i)] Assume $H\in(0,\frac12)$ and $H\beta<1,$ then
\begin{equation}\label{ld-e}
\lim_{t\to\infty}\frac1t\log \E\exp\left(\theta \L^\gamma_t(S^H)  \right)=\mathcal E_1(\gamma, H,\beta)\, \theta^{\frac{1}{1-H\beta}},  
\end{equation}
and consequently, 
\begin{equation}\label{ld-p}
\lim_{x\to\infty} \frac1{x^{\frac1{H\beta}}} \log \mathbb P(\L^\gamma_1(S^H)\ge x)=-C(\gamma,H,\beta), 
\end{equation}
where the constants $\mathcal E_1(\gamma, \alpha,\beta)$ and $C(\gamma,H,\beta)$ are given in Theorem \ref{thm-fbm}. 
\item[(ii)] Assume $H, K\in(0,1)$ and $HK\beta<1,$ then
\begin{equation}\label{ld-e'}
\lim_{t\to\infty}\frac1t\log \E\exp\left(\theta \L^\gamma_t(Z^{H,K})  \right)=2^{-\frac{(1-K)\beta}{2(1-HK\beta)}}\mathcal E_1(\gamma, HK,\beta)\, \theta^{\frac{1}{1-HK\beta}},  
\end{equation}
and consequently, 
\begin{equation}\label{ld-p'}
\lim_{x\to\infty} \frac1{x^{\frac1{HK\beta}} }\log \mathbb P(\L^\gamma_1(Z^{H,K})\ge x)=-2^{\frac{(1-K)\beta}{2HK\beta}}C(\gamma,HK,\beta).
\end{equation}

\end{itemize}
\end{theorem}
\begin{proof} Let us first prove \eqref{ld-e}.  By Propositions \ref{decom-sub} and \ref{prop-ld}, and together with the fact $\left(I_{0+}^{3/2}(L^2[0,1])\right)^d\subseteq \left(I_{0+}^{H+\frac12}(L^2[0,1])\right)^d$, it suffices to show that  for any $0<\e<1$, there exists a process  $\eta^\e$ such that  $\eta^\e_t=Y^H_t$ for $t\ge \e$, and  $\{\eta^\e_t, 0\leq t\leq 1\}$ belongs to $\left(I_{0+}^{3/2}(L^2[0,1])\right)^d$ almost surely.

  Note that $Y_t^H=\int_0^t M^H_s ds$, where $M^H_s=\int_0^\infty u^{\frac12-H}e^{-us}dW_u$ and $M^H$ has a version denoted again by $M^H$ that is infinitely differentiable on $(0,\infty)$ (see \cite[Theorem 1]{ln}). We construct $\eta^\e$ by following the proof of \cite[Proposition 3.5]{clrs} as follows: 
\begin{equation}\label{eq3.8}
\eta^\e_t=\begin{cases}
a_1 t^2+ a_2 t^3, & 0\le t\le \e,\\
Y^H_t, & t>\e,
\end{cases}
\end{equation}
where $a_1= 3\e^{-2} Y^H_\e- \e^{-1} M^H_\e$ and $a_2=-2\e^{-3}Y^H_\e+\e^{-2} M^H_\e.$ It is easy to verify that $\eta^\e$ and $(\eta^\e)'$ exist as continuous functions on $[0,1]$ with $\eta^\e_0=(\eta^\e)'_0=0$, and $(\eta^\e)''\in L^2[0,1]$. Then, by \cite[Proposition 3.4]{clrs},  we obtain $\{\eta^\e_t,t\in[0,1]\}\in \left(I_{0+}^{3/2}(L^2[0,1])\right)^d$ a.s. Noting that $\left(I_{0+}^{3/2}(L^2[0,1])\right)^d\subseteq\left( I^{H+\frac12}_{0+}(L^2[0,1])\right)^d=\mathbb{H}(B^H)$, we have $\{\eta^\e_t,t\in[0,1]\}\in \mathbb{H}(B^H)$ a.s.  Hence,  equation \eqref{ld-e} follows from Theorem \ref{thm-fbm} and Propositions \ref{prop-ld} and \ref{decom-sub}.

Next, we will prove \eqref{ld-e'} by an analogue to the proof of \eqref{ld-e}. Since  $\{Y_{t}^{K/2}=\int_0^t M^{K/2}_s ds,  t\geq 0\}$ is  infinitely differentiable on $(0,\infty)$ (see \cite[Theorem 1]{ln}), we obtain that 
 \[
 \tilde Y_t^{H,K} := Y_{t^{2H}}^{K/2}=\int_0^t2Hs^{2H-1}M^{K/2}_{s^{2H}}\, ds=:\int_0^t \tilde M_s^{H,K}ds 
 \]
 is  infinitely differentiable on $(0,\infty)$.
  Thus we can construct $\tilde \eta^\e$ in the same way as for $\eta^\e$ in \eqref{eq3.8} by replacing $Y^H$ with $\tilde Y^{H,K}$ and replacing $M^H$ with $\tilde M^{H,K}$. For the process $\tilde \eta^\e$ we have  $\tilde\eta^\e_t=\tilde Y_t^{H,K}$ for $t\ge \e$, and  almost surely $\{\tilde\eta^\e_t, t\in[0,1]\} \in \left(I_{0+}^{3/2}(L^2[0,1])\right)^d$ which belongs to  $\bh(Z^{H,K}) $ by Proposition \ref{prop-bi}.  Therefore, we can obtain \eqref{ld-e'} follows by Theorem \ref{thm-fbm} and  Propositions \ref{decom-bi} and \ref{prop-ld} and the homogeneity of $\gamma(\cdot)$. 
  
  Using the same argument as in the proof of equation \eqref{ld-p-1}, we can prove \eqref{ld-p} and  \eqref{ld-p'} respectively.   \hfill
\end{proof}

\begin{remark}
There are other examples of self-similar Gaussian processes $X$ for which $\L^\gamma_t(X)$ has large deviations. For instance, the Gaussian process $X$ with parameters $H\in(0,1)$ and $\lambda\in(0,H)$ introduced in \cite{HN} can be decomposed in law as the sum of a fBm (up to a factor of a constant) and the Gaussian process $Y$ defined in \eqref{eq-y} (up to a factor of a constant) independent of the fBm. Hence, using the same technique in dealing with sub-fBm, we can show the large deviations for $\L^\gamma_t(X)$. In the case $H\in(\frac12,1)$, the Gaussian process $X$ has the same law (up to a constant) as the process $u(t,0)$, where $u(t,x)$ is the solution to the following stochastic heat equation 
\[
\begin{cases}
\frac{\partial u}{\partial t}=\frac12 \Delta u+\dot{W}(t,x), \ t\geq 0, \ x\in\mathbb{R}^d,\\
u(0,x)=0,
\end{cases}
\]
with $\dot{W}$ being a zero mean Gaussian field with a covariance of the form
\[
\E\left(\dot{W}(t,x)\dot{W}(s,y)\right)=|s-t|^{2H-2}|x-y|^{-\beta},
\]
where $0<\beta<\min\{d,2\}$.
\end{remark}
\bigskip
As an application of the results in Theorems \ref{thm-fbm} and \ref{thm-sfbm}, we can show the following critical exponential integrability.
\begin{corollary}
\begin{itemize}
\item[(i)] Assume $H\in(0,1)$ and $H\beta<1$. We have 
\begin{itemize}
\item[(a)] when $p<\frac1{H\beta}$, $\E \exp\left( \lambda \left(\L^\gamma_1(B^H)\right)^p \right)<\infty$ for all $\lambda >0$;
\item[(b)] when $p>\frac1{H\beta}$, $\E \exp\left(  \lambda \left(\L^\gamma_1(B^H)\right)^p \right)=\infty$ for all $\lambda >0$;
\item[(c)] when $p=\frac1{H\beta}$, $\E \exp\left( \lambda \left(\L^\gamma_1(B^H)\right)^p\right)<\infty$ for $\lambda<C(\gamma, H,\beta)$, 

and $\E \exp\left( \lambda (L_1(B^H))^p \right)=\infty$ for $\lambda>C(\gamma, H,\beta)$,

where $C(\gamma,H,\beta)$ is given in \eqref{eq-C}.  
\end{itemize}
\item[(ii)] Assume $H\in(0,\frac12)$ and $H\beta<1$. We have 
\begin{itemize}
\item[(a)] when $p<\frac1{H\beta}$, $\E \exp\left( \lambda \left(\L^\gamma_1(S^H)\right)^p \right)<\infty$ for all $\lambda >0$;
\item[(b)] when $p>\frac1{H\beta}$, $\E \exp\left(  \lambda \left(\L^\gamma_1(S^H)\right)^p \right)=\infty$ for all $\lambda >0$;
\item[(c)] when $p=\frac1{H\beta}$, $\E \exp\left( \lambda \left(\L^\gamma_1(S^H)\right)^p\right)<\infty$ for $\lambda<C(\gamma, H,\beta)$,  

and $\E \exp\left( \lambda (L_1(S^H))^p \right)=\infty$ for $\lambda>C(\gamma, H,\beta)$.
\end{itemize}

\item[(iii)] Assume $H, K\in(0,1)$ and $HK\beta<1$. We have 
\begin{itemize}
\item[(a)] when $p<\frac1{HK\beta}$, $\E \exp\left( \lambda \left(\L^\gamma_1(Z^{H,K})\right)^p \right)<\infty$ for all $\lambda >0$;
\item[(b)] when $p>\frac1{HK\beta}$, $\E \exp\left( \lambda \left(\L^\gamma_1(Z^{H,K})\right)^p \right)=\infty$ for all $\lambda >0$;
\item[(c)] when $p=\frac1{HK\beta}$, $\E \exp\left( \lambda \left(\L^\gamma_1(Z^{H,K})\right)^p \right)<\infty$ for $\lambda<2^{\frac{p(1-K)\beta}{2}}C(\gamma,HK,\beta)$,

 and $\E \exp\left( \lambda \left(\L^\gamma_1(Z^{H,K})\right)^p \right)=\infty$ for $\lambda>2^{\frac{p(1-K)\beta}{2}}C(\gamma, HK,\beta)$.
\end{itemize}
\end{itemize}
\end{corollary}
\begin{proof}
By Fubini's theorem, we have
$$\E \exp\left( \lambda (\L^\gamma_1(B^H))^p \right)-1=\int_0^\infty \mathbb P\left(\L^\gamma_1(B^H)\ge (\lambda^{-1}y)^{1/p}  \right) e^{y}dy.$$
Then, the desired result for $\E \exp\left( \lambda (\L^\gamma_1(B^H))^p \right)$ follows from \eqref{ld-p-1}.

 Similarly, one can show the result for  $\E \exp\left( \lambda (\L_1^\gamma(S^H))^p \right)$ and $\E \exp\left( \lambda (\L_1^\gamma(Z^{H,K}))^p \right)$.  We complete the proof. \hfill
\end{proof}

\begin{remark} 
Analogous to the result on the local times of Riemann-Liouville process and fractional Brownian motion in \cite[Theorem 2.5]{clrs}, the large deviation results in Proposition \ref{prop-3.2} and Theorems \ref{thm-fbm} and \ref{thm-sfbm} can be applied to  argue the law of the iterated logarithm. Due to the consideration of the length of this article, the law of the iterated logarithm  will be discussed in a separate work.
\end{remark}
%\begin{remark}[On intersection local time and iterated law]
%I think it is possible to have a theorem analogous to \cite[Theorem 2.4]{clrs} using the same approach. But it is a bit lengthy to do this, and I don't think this will improve this article much.
%
% Regarding the iterated law \cite[Theorem 2.5]{clrs}, it is not clear to me how to get Lemma 6.1 for fBm case (2.15) and (2.17).  The problem is the norm of the RKHS in Lemma 6.1 is just for the Riemann-Liouville process $W^H_t=\int_0^t(t-s)^{H-1/2}dW_s$. As we know that the RHSP for $W^H$ and fBm $B^H$ are both $I^{H+1/2}(L^2[0,T])$, for any $f(t)=I_{0+}^{H+\frac12}(I_{0+}^{\frac12-H}(f))$ (e.g., $f(t)=\int_0^t g(s) ds$ with $g$ being continuous), $\| f\|_{\mathcal H (W^H)}=\|I_{0+}^{\frac12-H}(f)\|_{L^2[0,T]}$ (see also \cite[Proposition 3.4]{clrs}). But this does not hold for fBm, whose norm of RKHK should be more complex. 
%\end{remark}


\begin{thebibliography}{10}

\bibitem{al}
D.~Alpay and D.~Levanony.
\newblock On the reproducing kernel {H}ilbert spaces associated with the
  fractional and bi-fractional {B}rownian motions.
\newblock {\em Potential Anal.}, 28(2):163--184, 2008.

\bibitem{a}
N.~Aronszajn.
\newblock Theory of reproducing kernels.
\newblock {\em Trans. Amer. Math. Soc.}, 68:337--404, 1950.

\bibitem{bcr}
R.~Bass, X.~Chen, and J.~Rosen.
\newblock Large deviations for {R}iesz potentials of additive processes.
\newblock {\em Ann. Inst. Henri Poincar\'{e} Probab. Stat.}, 45(3):626--666,
  2009.

\bibitem{Berman}
S.~M. Berman.
\newblock Local times and sample function properties of stationary {G}aussian
  processes.
\newblock {\em Trans. Amer. Math. Soc.}, 137:277--299, 1969.

\bibitem{berman}
S.~M. Berman.
\newblock Local nondeterminism and local times of {G}aussian processes.
\newblock {\em Indiana Univ. Math. J.}, 23:69--94, 1973/74.


\bibitem{bhoz}
F.~Biagini, Y.~Hu, B.~\O~ksendal, and T.~Zhang.
\newblock {\em Stochastic calculus for fractional {B}rownian motion and
  applications}.
\newblock Probability and its Applications (New York). Springer-Verlag London,
  Ltd., London, 2008.

\bibitem{BGT}
T.~Bojdecki, L.~G. Gorostiza, and A.~Talarczyk.
\newblock Sub-fractional {B}rownian motion and its relation to occupation
  times.
\newblock {\em Statist. Probab. Lett.}, 69(4):405--419, 2004.

\bibitem{chen07}
X.~Chen.
\newblock Large deviations and laws of the iterated logarithm for the local
  times of additive stable processes.
\newblock {\em Ann. Probab.}, 35(2):602--648, 2007.

\bibitem{chen}
X.~Chen.
\newblock {\em Random walk intersections}, volume 157 of {\em Mathematical
  Surveys and Monographs}.
\newblock American Mathematical Society, Providence, RI, 2010.
\newblock Large deviations and related topics.

\bibitem{clrs}
X.~Chen, W.~V. Li, J.~Rosi\'{n}ski, and Q.-M. Shao.
\newblock Large deviations for local times and intersection local times of
  fractional {B}rownian motions and {R}iemann-{L}iouville processes.
\newblock {\em Ann. Probab.}, 39(2):729--778, 2011.

\bibitem{du}
L.~Decreusefond and A.~S. \"{U}st\"{u}nel.
\newblock Stochastic analysis of the fractional {B}rownian motion.
\newblock {\em Potential Anal.}, 10(2):177--214, 1999.

\bibitem{HN}
D.~Harnett and D.~Nualart.
\newblock Decomposition and limit theorems for a class of self-similar
  {G}aussian processes.
\newblock In {\em Stochastic analysis and related topics}, volume~72 of {\em
  Progr. Probab.}, pages 99--116. Birkh\"{a}user/Springer, Cham, 2017.

\bibitem{Janson}
S.~Janson.
\newblock {\em Gaussian {H}ilbert spaces}, volume 129 of {\em Cambridge Tracts
  in Mathematics}.
\newblock Cambridge University Press, Cambridge, 1997.

\bibitem{km02}
W.~K\"{o}nig and P.~M\"{o}rters.
\newblock Brownian intersection local times: upper tail asymptotics and thick
  points.
\newblock {\em Ann. Probab.}, 30(4):1605--1656, 2002.

\bibitem{ln}
P.~Lei and D.~Nualart.
\newblock A decomposition of the bifractional {B}rownian motion and some
  applications.
\newblock {\em Statist. Probab. Lett.}, 79(5):619--624, 2009.

\bibitem{ls}
W.~V. Li and Q.-M. Shao.
\newblock Gaussian processes: inequalities, small ball probabilities and
  applications.
\newblock In {\em Stochastic processes: theory and methods}, volume~19 of {\em
  Handbook of Statist.}, pages 533--597. North-Holland, Amsterdam, 2001.

\bibitem{luan}
N.~Luan.
\newblock Strong local non-determinism of sub-fractional {B}rownian motion,
  2015.

\bibitem{M-VN}
B.~B. Mandelbrot and J.~W. Van~Ness.
\newblock Fractional {B}rownian motions, fractional noises and applications.
\newblock {\em SIAM Rev.}, 10:422--437, 1968.

\bibitem{nualart06}
D.~Nualart.
\newblock {\em The {M}alliavin calculus and related topics}.
\newblock Probability and its Applications (New York). Springer-Verlag, Berlin,
  second edition, 2006.

\bibitem{ct}
J.~Ruiz~de Ch\'{a}vez and C.~Tudor.
\newblock A decomposition of sub-fractional {B}rownian motion.
\newblock {\em Math. Rep. (Bucur.)}, 11(61)(1):67--74, 2009.

\bibitem{skm}
S.~G. Samko, A.~A. Kilbas, and O.~I. Marichev.
\newblock {\em Fractional integrals and derivatives}.
\newblock Gordon and Breach Science Publishers, Yverdon, 1993.
\newblock Theory and applications, Edited and with a foreword by S. M.
  Nikol'ski\u{\i}, Translated from the 1987 Russian original, Revised by the
  authors.

\bibitem{st}
G.~Samorodnitsky and M.~S. Taqqu.
\newblock {\em Stable non-{G}aussian random processes}.
\newblock Stochastic Modeling. Chapman \& Hall, New York, 1994.
\newblock Stochastic models with infinite variance.

\bibitem{tx}
C.~A. Tudor and Y.~Xiao.
\newblock Sample path properties of bifractional {B}rownian motion.
\newblock {\em Bernoulli}, 13(4):1023--1052, 2007.

\bibitem{vdv-vz}
A.~W. van~der Vaart and J.~H. van Zanten.
\newblock Reproducing kernel {H}ilbert spaces of {G}aussian priors.
\newblock In {\em Pushing the limits of contemporary statistics: contributions
  in honor of {J}ayanta {K}. {G}hosh}, volume~3 of {\em Inst. Math. Stat. (IMS)
  Collect.}, pages 200--222. Inst. Math. Statist., Beachwood, OH, 2008.\\


\begin{tabular}{lll}
Xiaoming Song \\
Department of Mathematics\\
Drexel University\\ 
Philadelphia, PA 19104\\
{\tt xs73@drexel.edu}
\end{tabular}
\end{thebibliography}
\end{document}